\documentclass[11pt]{amsart}
\usepackage{amsmath}
\usepackage{amsthm}
\usepackage{amsfonts}
\usepackage{amssymb}
\usepackage{pgf,tikz}
\usepackage{tikz-cd}
\usepackage{hyperref}

\newcommand{\meet}{\land}
\newcommand{\join}{\lor}
\newcommand{\up}{{\sf Up}}
\newcommand{\boxp}{\blacksquare_P}
\newcommand{\boxf}{\blacksquare_F}
\newcommand{\diap}{\blacklozenge_P}
\newcommand{\diaf}{\blacklozenge_F}
\newcommand{\boxq}{\Box_Q}
\newcommand{\boxr}{\Box_R}
\newcommand{\qinv}{{Q  \raisebox{-2pt}{\scalebox{1.2}{$\breve{\hspace{3.5pt}}$}}}}
\newcommand{\rinv}{{R  \raisebox{-2pt}{\scalebox{1.2}{$\breve{\hspace{3.5pt}}$}}}}
\newcommand{\boxqinv}{\Box_{Q  \raisebox{-4.3pt}{\scalebox{1.2}{$\breve{\hspace{3.5pt}}$}}}}
\newcommand{\boxrinv}{\Box_{R  \raisebox{-4.3pt}{\scalebox{1.2}{$\breve{\hspace{3.5pt}}$}}}}
\newcommand{\diaq}{\Diamond_Q}
\newcommand{\diar}{\Diamond_R}
\newcommand{\diaqinv}{\Diamond_{Q  \raisebox{-4.3pt}{\scalebox{1.2}{$\breve{\hspace{3.5pt}}$}}}}
\newcommand{\ts}{{\sf TS4}}
\newcommand{\boxpr}{\Box_P}
\newcommand{\boxfr}{\Box_F}
\newcommand{\diapr}{\Diamond_P}

\newcommand{\boxprdot}{\Box_P'}
\newcommand{\boxfrdot}{\Box_F'}
\newcommand{\diaprdot }{\Diamond_P'}
\newcommand{\diafrdot }{\Diamond_F'}
\newcommand{\mst}{{\sf MS4.t}}
\newcommand{\ms}{{\sf MS4}}
\newcommand{\mipc}{{\sf MIPC}}
\newcommand{\ipc}{{\sf IPC}}
\newcommand{\mha}{{\sf MHA}}
\newcommand{\iqc}{{\sf IQC}}
\newcommand{\qsfour}{{\sf QS4}}
\newcommand{\qsfourt}{{\sf QS4.t}}
\newcommand{\sfour}{{\sf S4}}
\newcommand{\sfive}{{\sf S5}}

\newcommand{\st}{{\sf S4.t}}
\newcommand{\qst}{{\sf QS4.t}}
\newcommand{\qost}{{\sf Q^\circ S4.t}}

\newcommand{\qeinv}{{Q \raisebox{-2pt}{\scalebox{1.2}{$\breve{\hspace{3.5pt}}$}}}_{\hspace{-4pt} E}}

\newtheorem{theorem}{Theorem}[section]
\newtheorem*{theorem*}{Theorem}
\newtheorem{lemma}[theorem]{Lemma}
\newtheorem{proposition}[theorem]{Proposition}
\newtheorem*{proposition*}{Proposition}
\newtheorem{corollary}[theorem]{Corollary}
\newtheorem*{corollary*}{Corollary}

\theoremstyle{definition}
\newtheorem{definition}[theorem]{Definition}
\newtheorem*{definition*}{Definition}

\newtheorem*{example*}{Example}

\newtheorem{remark}[theorem]{Remark}
\newtheorem*{remark*}{Remark}

\author[G. Bezhanishvili]{Guram Bezhanishvili}
\address{Department of Mathematical Sciences\\
New Mexico State University\\
Las Cruces NM 88003\\
USA}
\email{guram@nmsu.edu}

\author[L. Carai]{Luca Carai}
\address{Department of Mathematical Sciences\\
New Mexico State University\\
Las Cruces NM 88003\\
USA}
\email{lcarai@nmsu.edu}

\date{}

\title[Temporal interpretation of monadic intuitionistic quantifiers]{Temporal interpretation of intuitionistic quantifiers: Monadic case}



\begin{document}

\begin{abstract}
In a recent paper we showed that intuitionistic quantifiers admit the following temporal interpretation: ``always in the future" (for $\forall$) and ``sometime in the past" (for $\exists$). In this paper we study this interpretation for the monadic fragment $\sf MIPC$ of the intuitionistic predicate logic. It is well known that $\sf MIPC$ is translated fully and faithfully into the monadic fragment $\sf MS4$ of the predicate $\sf S4$ (G\"{o}del translation). We introduce a new tense extension of $\sf S4$, denoted by $\sf TS4$, and provide an alternative full and faithful translation of $\sf MIPC$ into $\sf TS4$, which yields the temporal interpretation of monadic intuitionistic quantifiers mentioned above. We compare this new translation with the G\"{o}del translation by showing that both $\sf MS4$ and $\sf TS4$ can be translated fully and faithfully into a tense extension of $\sf MS4$, which we denote by $\sf MS4.t$. This is done by utilizing the algebraic and relational semantics for the new logics introduced. As a byproduct, we prove the finite model property (fmp) for $\sf MS4.t$ and show that the fmp for the other logics involved can be derived as a consequence of the fullness and faithfulness of the translations considered.
\end{abstract}

\maketitle

\section{Introduction}\label{sec:introduction}

It is well known that, unlike classical quantifiers, the interpretation of intuitionistic quantifiers is non-symmetric in that $\forall x A$ is true at a world $w$ iff $A$ is true at every object $a$ in the domain $D_v$ of every world $v$ accessible from $w$, and $\exists x A$ is true at $w$ iff $A$ is true at some object $a$ in the domain $D_w$ of $w$.
This non-symmetry is also evident in the G\"{o}del translation of the intuitionistic predicate logic $\iqc$ into the predicate $\sfour$, denoted $\qsfour$, since $\forall xA$ is translated as $\Box\forall x A^t$ and $\exists x A$
as $\exists x A^t$, where $A^t$ is the translation of $A$.
Because of this, it is common to give a temporal interpretation of the intuitionistic universal quantifier as ``always in the future." In \cite{BC20b} we showed that it is also possible to give a temporal interpretation of the intuitionistic existential quantifier as ``sometime in the past."

In this paper we concentrate on the monadic (one-variable) fragment of
$\iqc$. It is well known that this fragment is axiomatized by Prior's monadic intuitionistic propositional calculus $\mipc$ \cite{Bul66,OS88}. The monadic fragment of
$\qsfour$ was studied by Fischer-Servi \cite{FS77} who showed that the G\"{o}del translation of
$\iqc$
into
$\qsfour$
restricts to the monadic case. We denote this monadic fragment
by $\ms$, introduce a tense counterpart of it, which we denote by $\ts$, modify the G\"{o}del translation, and prove that it embeds $\mipc$ into $\ts$ fully and faithfully. This allows us to give the desired temporal interpretation of intuitionistic monadic quantifiers as ``always in the future" (for $\forall$) and ``sometime in the past" (for $\exists$).

While $\ms$ and $\ts$ are not comparable, we introduce a common extension, which we denote by $\mst$. The system $\mst$ can be thought of as a tense extension of $\ms$. We prove that there exist full and faithful translations of $\mipc$, $\ms$, and $\ts$ into $\mst$, yielding the following diagram, which commutes up to logical equivalence. In the diagram, the G\"{o}del translation is denoted by $(\:)^t$, our new translation by $(\:)^\natural$, and the three translations into $\mst$ by $(\:)^\flat$, $(\:)^\#$ and $(\:)^\dagger$, respectively.
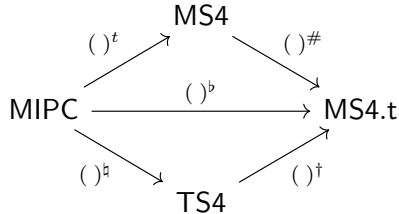
\[
\begin{tikzcd}
& \ms \arrow[rd, "( \; )^\#"] & \\
\mipc \arrow[ru, " ( \; )^t"] \arrow[rd, "(\; )^\natural"']  \arrow[rr,  " ( \; )^\flat"]& & \mst \\
& \ts \arrow[ru, "(\; )^\dagger"'] &
\end{tikzcd}
\]

We prove these results by utilizing the algebraic and relational semantics, and by showing that each of these systems is canonical.
In addition, we prove that $\mst$ has the fmp. It is then an easy consequence of the fullness and faithfulness of the translations considered that the other systems also have the fmp. That $\mipc$ has the fmp was first proved by Bull \cite{Bul65}, and an error in the proof was corrected independently by Fischer-Servi \cite{FS78a} and Ono \cite{Ono77}. To the best of our knowledge, the proof of the fmp for $\ts$ (and possibly also for $\ms$) is new. We conclude the paper by comparing the above translations with the translation of
$\iqc$ into a version of predicate $\st$ studied in \cite{BC20b}.

\section{Translation of $\mipc$ into $\ms$}\label{sec:godel translation}

In this preliminary section we briefly recall the syntax and semantics of $\mipc$ and $\ms$, and give an alternate proof that the
G\"{o}del translation of $\mipc$ into $\ms$ is full and faithful.

\subsection{$\mipc$}

We start by recalling the definition of Prior's monadic intuitionistic propositional calculus $\mipc$. Let $\mathcal L$ be a propositional language
and let $\mathcal L_{\forall\exists}$ be an extension of $\mathcal L$ with two modalities $\forall$ and $\exists$.

\begin{definition}
The \textit{monadic intuitionistic propositional calculus} $\mipc$ is the intuitionistic modal logic in the propositional modal language
$\mathcal L_{\forall\exists}$ containing
\begin{enumerate}
\item all theorems of the intuitionistic propositional calculus $\ipc$;
\item the $\sfour$-axioms for $\forall$:
\begin{enumerate}
\item $\forall(p\land q)\leftrightarrow(\forall p\land\forall q)$,
\item $\forall p \rightarrow p$,
\item $\forall p \rightarrow \forall \forall p$;
\end{enumerate}
\item the $\sfive$-axioms for $\exists$:
\begin{enumerate}
\item $\exists(p\vee q)\leftrightarrow(\exists p\vee\exists q)$,
\item $p \rightarrow \exists p$,
\item $\exists \exists p \rightarrow \exists p$,
\item $(\exists p \land \exists q) \rightarrow \exists (\exists p \land q)$;
\end{enumerate}
\item the axioms connecting $\forall$ and $\exists$:
\begin{enumerate}
\item $\exists\forall p\leftrightarrow\forall p$,
\item $\exists p \leftrightarrow \forall\exists p$;
\end{enumerate}
\end{enumerate}
and closed under the rules of modus ponens, substitution, and necessitation $(\varphi / \forall \varphi )$.
\end{definition}

\begin{remark}
There are a number of axioms that are equivalent to the axiom (3d) (see, e.g., \cite[Lem.~2(d)]{Bez98}).
\end{remark}

The algebraic semantics for $\mipc$ is given by monadic Heyting algebras. These algebras were first introduced by Monteiro and Varsavsky \cite{MV57} as a generalization of monadic (boolean) algebras of Halmos \cite{Hal55}. For a detailed study of monadic Heyting algebras we refer to \cite{Bez98,Bez99,Bez00}.

\begin{definition}\label{def:interior and closure}
Let $H$ be a Heyting algebra.
\begin{enumerate}
\item A unary function ${\sf i}:H\to H$ is an \emph{interior operator} on $H$ if
\begin{enumerate}
\item ${\sf i}(a\wedge b) = {\sf i}a \wedge {\sf i}b$,
\item ${\sf i}1 =1$,
\item ${\sf i}a \leq a$,
\item ${\sf i}a \leq {\sf ii}a$.
\end{enumerate}
\item A unary function ${\sf c}:H\to H$ is a \emph{closure operator} on $H$ if
\begin{enumerate}
\item ${\sf c}(a\vee b) = {\sf c}a \vee {\sf c}b$,
\item ${\sf c}0 =0$,
\item $a \leq {\sf c}a$,
\item ${\sf cc}a \leq {\sf c}a$.
\end{enumerate}
\end{enumerate}
\end{definition}

\begin{definition} \label{def:mha}
A \textit{monadic Heyting algebra} is a triple $\mathfrak A=(H,\forall,\exists)$ where $H$ is a Heyting algebra, $\forall$ is an interior operator
on $H$, and $\exists$ is a closure operator on $H$ satisfying:
\begin{enumerate}
\item $\exists (\exists a \meet b)= \exists a \meet \exists b$,
\item $\forall \exists a = \exists a$,
\item $\exists \forall a = \forall a$.
\end{enumerate}
Let $\mha$ be the class of all monadic Heyting algebras.
\end{definition}

\begin{remark} \label{rem:adjoints}
Let $(H,\forall,\exists)$ be a monadic Heyting algebra.
\begin{enumerate}
\item There are a number of equivalent conditions to Definition~\ref{def:mha}(1) (see, e.g., \cite[Lem.~2(d)]{Bez98}). These together with
the conditions connecting $\forall$ and $\exists$ yield that the fixpoints of $\forall$ form a subalgebra $H_0$ of $H$ which coincides
with the subalgebra of the fixpoints of $\exists$. Moreover, $\forall$ and $\exists$ are the right and left adjoints of the embedding
$H_0\to H$, and up to isomorphism each monadic Heyting algebra arises this way (see, e.g., \cite[Sec.~3]{Bez98}).
\item The non-symmetry of $\forall$ and $\exists$ is manifested by the fact that the $\forall$-analogue
$\forall (\forall a \join b)= \forall a \join \forall b$ of Definition~\ref{def:mha}(1) does not hold in general.
\end{enumerate}
\end{remark}

The standard Lindenbaum-Tarski construction (see, e.g., \cite{RS63}) yields that monadic Heyting algebras provide a sound and complete algebraic
semantics for $\mipc$.

We next turn to the relational semantics for $\mipc$. There are several such (see, e.g., \cite{Bez99}), but we concentrate on the one introduced by Ono \cite{Ono77}.

\begin{definition}\label{def:ono}
An \textit{$\mipc$-frame} is a triple $\mathfrak F=(X,R,Q)$ where $X$ is a set, $R$ is a partial order, $Q$ is a quasi-order
(reflexive and transitive), and the following two conditions are satisfied:
\begin{enumerate}
\item[(O1)] $R \subseteq Q$,
\item[(O2)] $x Q y \Rightarrow (\exists z)(x R z \; \& \; z E_Q y)$.
\end{enumerate}
Here $E_Q$ is the equivalence relation defined by $x E_Q y$ iff $x Q y$ and $y Q x$.
\end{definition}

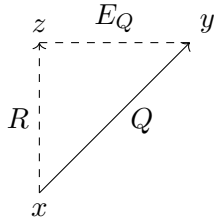
\begin{figure}[!ht]
\begin{center}
\begin{tikzpicture}
\node [below] at (0,0) {$x$};
\node [above] at (0,2) {$z$};
\node [above right] at (2,2) {$y$};
\draw [dashed, ->] (0,0) -- (0,2);
\draw [->] (0,0) -- (2,2);
\draw [dashed] (0,2) -- (2,2);
\node [above] at (1,2) {$E_Q$};
\node [left] at (0,1) {$R$};
\node [left] at (1.65,0.96) {$Q$};
\end{tikzpicture}
\end{center}
\caption{Condition (O2).}
\end{figure}

Let $\mathfrak F=(X,R,Q)$ be an $\mipc$-frame. As usual, for $x\in X$, we write
\[
R[x]=\{y\in X\mid xRy\} \mbox{ and } R^{-1}[x]=\{y\in X\mid yRx\},
\]
and for $U\subseteq X$, we write
\[
R[U]=\bigcup\{R[u]\mid u\in U\} \mbox{ and } R^{-1}[U]=\bigcup\{R^{-1}[u]\mid u\in U\}.
\]
We use the same notation for $Q$ and $E_Q$. Since $E_Q$ is an equivalence relation, we have that $E_Q[x]=(E_Q)^{-1}[x]$ and $E_Q[U]=(E_Q)^{-1}[U]$.

We call a subset $U$ of $X$ an \emph{$R$-upset} provided $U=R[U]$ ($x\in U$ and $xRy$ imply $y\in U$). Let $\up(X)$ be the set of all $R$-upsets
of $\mathfrak F$. It is well known that $\up(X)$ is a Heyting algebra, where the lattice operations are set-theoretic union and intersection, and
$U\to V$ is calculated by
\[
U\to V = \{x\in X \mid R[x]\cap U\subseteq V\} = X \setminus R^{-1}[U \setminus V].
\]
In addition, for $U\in\up(X)$, define
\[
\forall_Q (U) = X \setminus Q^{-1}[X \setminus U] \mbox{ and } \exists_Q (U) = E_Q[U].
\]
Then $\mathfrak F^+=(\up(X),\forall_Q,\exists_Q)$ is a monadic Heyting algebra (see, e.g., \cite[Sec.~6]{Bez99}).

\begin{remark}\label{rem:E_Q and Q on R-upsets}
If $U \in \up(X)$, then Definition~\ref{def:ono}(O2) implies that $E_Q[U]=Q[U]$. That $\exists_Q (U)=Q[U]$ motivates our interpretation of $\exists$ as ``sometime in the past.'' Indeed, taking $Q[U]$ is the standard way to associate an operator on $\wp(X)$ to the tense modality ``sometime in the past'' (see, e.g., \cite[p.~151]{Tho72}). As a consequence of this, $(\mathfrak F^+)_0$ is the set of $Q$-upsets of $\mathfrak{F}$.
\end{remark}

Each monadic Heyting algebra $\mathfrak A=(H,\forall,\exists)$ can be represented as a subalgebra of $\mathfrak F^+$ for some $\mipc$-frame
$\mathfrak F$. For this we recall the definition of the canonical frame of $\mathfrak A$.

\begin{definition}
Let $\mathfrak A=(H,\forall,\exists)$ be a monadic Heyting algebra. The \emph{canonical frame} of $\mathfrak A$ is the frame
$\mathfrak A_+=(X_{\mathfrak A},R_{\mathfrak A},Q_{\mathfrak A})$ where $X_{\mathfrak A}$ is the set of prime filters of $H$,
$R_{\mathfrak A}$ is the inclusion relation, and $x Q_{\mathfrak A} y$ iff $x\cap H_0 \subseteq y$ (equivalently, $x\cap H_0 \subseteq y\cap H_0$).
\end{definition}

By \cite[Sec.~6]{Bez99}, $\mathfrak A_+$ is an $\mipc$-frame.

\begin{definition}
We call an $\mipc$-frame $\mathfrak F$ \emph{canonical} if it is isomorphic to $\mathfrak{A}_+$ for some monadic Heyting algebra $\mathfrak{A}$.
\end{definition}

Define the \emph{Stone map} $\beta : \mathfrak A \to \up(X_{\mathfrak A})$ by
\[
\beta(a)=\{ x \in X_{\mathfrak A} \mid a \in x \}.
\]
By \cite[Sec.~6]{Bez99}, $\beta:\mathfrak A\to(\mathfrak A_+)^+$ is a one-to-one homomorphism of monadic Heyting algebras. Thus, we arrive
at the following representation theorem for monadic Heyting algebras.

\begin{proposition}
Each monadic Heyting algebra $\mathfrak A$ is isomorphic to a subalgebra of $(\mathfrak A_+)^+$.
\end{proposition}

\begin{remark}
\begin{enumerate}
\item[]
\item The image of $\mathfrak A$ inside $(\mathfrak A_+)^+$ can be recovered by introducing a Priestley topology on $X_\mathfrak{A}$. This leads to the notion of \emph{perfect $\mipc$-frames} and a duality between the category of monadic Heyting algebras and the category of perfect $\mipc$-frames; see~\cite[Thm.~17]{Bez99}.
\item When $\mathfrak{A}$ is finite, its embedding into $(\mathfrak A_+)^+$ is an isomorphism, and hence the categories of finite monadic Heyting algebras and finite $\mipc$-frames are dually equivalent.
\end{enumerate}
\end{remark}

The next corollary is an immediate consequence of the above considerations.

\begin{corollary}\label{prop:canonical}
$\mipc$ is canonical; that is,
\[
\mathfrak A\in\mha \Rightarrow (\mathfrak A_+)^+\in\mha.
\]
\end{corollary}

A \emph{valuation} on an $\mipc$-frame $\mathfrak F=(X,R,Q)$ is a map $v$ associating an $R$-upset of $X$ to any propositional letter of $\mathcal L_{\forall\exists}$. The connectives $\wedge,\vee,\to,\neg$ are then interpreted as in intuitionistic Kripke frames, and $\forall,\exists$ are interpreted by
\begin{equation*}
\begin{array}{l c l}
x \vDash_v \forall \varphi & \text{ iff } & (\forall y \in X)(x Q y \Rightarrow y \vDash_v \varphi ), \\
x \vDash_v \exists \varphi & \text{ iff } & (\exists y \in X)(x E_Q y \;\& \; y \vDash_v \varphi ).
\end{array}
\end{equation*}

As usual, we say that $\varphi$ is \emph{valid} in $\mathfrak F$, and write $\mathfrak F \vDash \varphi$, if $x \vDash_v \varphi$ for every
valuation $v$ and every $x \in X$.

Soundness of $\mipc$ with respect to this semantics is straightforward to prove. For completeness, it is sufficient to utilize the
algebraic completeness and the representation theorem for monadic Heyting algebras. As a result, we arrive at the following:

\begin{theorem}\label{thm:relational semantics for mipc}
$\mipc \vdash \varphi$ iff $\mathfrak F \vDash \varphi$ for every $\mipc$-frame $\mathfrak F$.
\end{theorem}

We conclude this section by recalling that $\mipc$ has the fmp. This was first established by Bull \cite{Bul66} using algebraic semantics.
His proof contained a gap, which was corrected independently by Fischer-Servi \cite{FS78a} and Ono \cite{Ono77}. A semantic proof is given
in \cite{GKWZ03}, which is based on the technique developed by Grefe \cite{Gre98}. We will give yet another proof of this result in
Section~\ref{sec:FMP}.

\subsection{MS4}

Let $\mathcal L_{\Box\forall}$ be a propositional bimodal language with two modal operators $\Box$ and $\forall$.

\begin{definition}
The \emph{monadic $\sf S4$}, denoted $\ms$, is the smallest classical bimodal logic containing the $\sf S4$-axioms for $\Box$, the $\sf S5$-axioms
for $\forall$, the left commutativity axiom
\[
\Box \forall p \to \forall \Box p,
\]
and closed under modus ponens, substitution, $\Box$-necessitation, and $\forall$-necessi\-tation.
\end{definition}

As usual, $\Diamond$ is an abbreviation for $\neg\Box\neg$ and $\exists$ is an abbreviation for $\neg\forall\neg$.

\begin{remark}
Recalling the definition of fusion of two logics (see \cite{GKWZ03}), $\ms$ is obtained from the fusion ${\sf S4}\otimes{\sf S5}$
by adding the left commutativity axiom $\Box \forall p \rightarrow \forall \Box p$ which is the monadic version of the converse Barcan formula. The monadic version of the Barcan formula is the right commutativity axiom $\forall \Box p \rightarrow \Box \forall p$. Adding it to $\ms$ yields the product logic $\sf S4 \times \sf S5$; see~\cite[Ch.~5]{GKWZ03} for details.
\end{remark}

The algebraic semantics for $\ms$ is given by monadic $\sfour$-algebras. To define these algebras, we first recall the definition of $\sfour$-algebras and $\sfive$-algebras.

\begin{definition}
\begin{enumerate}
\item[]
\item An \emph{$\sfour$-algebra}, or an \emph{interior algebra}, is a pair $\mathfrak{B}=(B, \Box)$ where $B$ is a boolean algebra and $\Box$ is an interior operator on $B$ (see Definition~\ref{def:interior and closure}(1)).
\item An \emph{$\sfive$-algebra}, or a \emph{monadic algebra}, is an $\sfour$-algebra $\mathfrak{B}=(B, \forall)$ that in addition satisfies $a \leq \forall \exists a$ for all $a \in B$.
\end{enumerate}
\end{definition}

We are ready to define monadic $\sfour$-algebras.

\begin{definition}\label{def:ms-alg}
A \textit{monadic $\sfour$-algebra}, or an \emph{$\ms$-algebra} for short, is a tuple $\mathfrak B=(B, \Box, \forall)$ where
\begin{enumerate}
\item $(B, \Box)$ is an $\sfour$-algebra,
\item $(B, \forall)$ is an $\sfive$-algebra,
\item $\Box \forall a \leq \forall \Box a$ for each $a\in B$.
\end{enumerate}
\end{definition}

\begin{lemma}\label{lem:equivalent axioms ms-alg}
The axiom $\Box \forall a \leq \forall \Box a$ in Definition~\ref{def:ms-alg} can be replaced by any of the following:
\begin{enumerate}
\item $\Box \forall \Box a = \Box \forall a$.
\item $\forall \Box \forall a= \Box \forall a$.
\item $\exists \Box \exists a= \Box \exists a$.
\item $\Box \exists \Box a = \exists \Box a$.
\item $\exists \Box a \le \Box \exists a$.
\end{enumerate}
\end{lemma}

\begin{proof}
Showing that (1) and (2) are equivalent to $\Box \forall a \leq \forall \Box a$ is straightforward. That (3) and (4) are equivalent to (5) can be proved similarly (see~\cite{CarThesis} for details). We show that (2) and (3) are equivalent. Suppose (2) holds. Then for each $a \in B$, we have
\[
\forall \Box \exists a =\forall \Box \forall \exists a= \Box \forall \exists a = \Box \exists a.
\]
Using $\forall \Box \exists a = \Box \exists a$ twice, we obtain
\[
\exists \Box \exists a = \exists \forall \Box \exists a=  \forall \Box \exists a=\Box \exists a,
\]
yielding (3). Proving (2) from (3) is analogous.
\end{proof}

\begin{remark}
As noted above, the inequality $\Box \forall a \leq \forall \Box a$ is equivalent to the equality $\forall \Box \forall a = \Box \forall a$. This yields that the set $B_0$ of $\forall$-fixpoints of an $\ms$-algebra $(B, \Box, \forall)$ forms an $\sfour$-subalgebra of $(B, \Box)$ such that $\forall$ is the right adjoint to the embedding $B_0 \to B$. Moreover, up to isomorphism each $\ms$-algebra arises this way. This is similar to the case of monadic Heyting algebras (see Remark~\ref{rem:adjoints}).
\end{remark}

The Lindenbaum-Tarski construction yields that $\ms$-algebras provide a sound and complete algebraic semantics for $\ms$.

The relational semantics for $\ms$ was first introduced by Esakia \cite{Esa88}.

\begin{definition}\label{def:ms-frame}
An \textit{$\ms$-frame} is a triple $\mathfrak F=(X,R,E)$ where $X$ is a set, $R$ is a quasi-order, $E$ is an equivalence relation, and
the following commutativity condition is satisfied:
\begin{equation}\label{E}
\tag{E}
(\forall x,y,z \in X)(x E y \; \& \; y R z) \Rightarrow (\exists u \in X)(x R u \; \& \; u E z).
\end{equation}
\end{definition}

\begin{figure}[!ht]
\begin{center}
\begin{tikzpicture}
\node [left] at (0,0) {$x$};
\node [left] at (0,2) {$u$};
\node [right] at (2,0) {$y$};
\node [right] at (2,2) {$z$};
\draw [dashed, ->] (0,0) -- (0,2);
\draw [->] (2,0) -- (2,2);
\draw [dashed] (0,2) -- (2,2);
\draw  (0,0) -- (2,0);
\node [above] at (1,2) {$E$};
\node [above] at (1,0) {$E$};
\node [left] at (0,1) {$R$};
\node [right] at (2,1) {$R$};
\end{tikzpicture}
\end{center}
\caption{Condition (E).}
\end{figure}
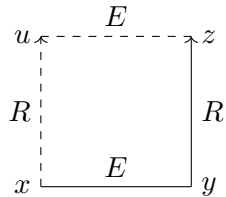

For an $\ms$-frame $\mathfrak F=(X,R,E)$, let $\wp(X)$ be the powerset of $X$ and for $U\in\wp(X)$ let
\[
\Box_R(U)=X\setminus R^{-1}[X\setminus U] \mbox{ and } \forall_E(U)=X\setminus E[X\setminus U].
\]
Since $R$ is a quasi-order, $(\wp(X),\Box_R)$ is an $\sfour$-algebra; and since $E$ is an equivalence relation, $(\wp(X),\forall_E)$ is an $\sfive$-algebra (see~\cite[Thm.~3.5]{JT51}). In addition, the commutativity condition yields that $\mathfrak F^+:=(\wp(X),\Box_R,\forall_E)$ is an $\ms$-algebra.

In fact, as in the case of monadic Heyting algebras, each $\ms$-algebra $\mathfrak B=(B,\Box,\forall)$ is isomorphic to a subalgebra of
$\mathfrak F^+$ for some $\ms$-frame $\mathfrak F$. We can take $\mathfrak F$ to be the canonical frame of $\mathfrak B$. Let $H$ be the set of $\Box$-fixpoints and $B_0$ the set of $\forall$-fixpoints. Then $H$ is a Heyting algebra which is a bounded sublattice of $B$, and $B_0$ is an $\sfour$-subalgebra of $(B, \Box)$.

\begin{remark} \label{rem:B_0 in F^+ for ms}
If $\mathfrak B=\mathfrak F^+$, then the elements of $H$ are the $R$-upsets of $\mathfrak{F}$ and the elements of $B_0$ are the $E$-saturated subsets of $\mathfrak{F}$ (that is, unions of $E$-equivalence classes).
\end{remark}

\begin{definition}\label{def:esakia}
Let $\mathfrak B=(B,\Box,\forall)$ be an $\ms$-algebra. The \emph{canonical frame} of $\mathfrak B$ is the frame
$\mathfrak B_+=(X_{\mathfrak B},R_{\mathfrak B},E_{\mathfrak B})$ where
$X_{\mathfrak B}$ is the set of ultrafilters of $B$, $x R_{\mathfrak B} y$ iff
$x \cap H \subseteq y$ (equivalently, $x \cap H \subseteq y \cap H$), and $x E_{\mathfrak B} y$ iff $x\cap B_0=y\cap B_0$.
\end{definition}

\begin{lemma}\label{lem:canonical frame for ms}
If $\mathfrak B$ is an $\ms$-algebra, then $\mathfrak B_+$ is an $\ms$-frame.
\end{lemma}

\begin{proof}
Since $(B,\Box)$ is an $\sfour$-algebra, $R_{\mathfrak B}$ is a quasi-order (see~\cite[Thm.~3.14]{JT51}); and since $(B,\forall)$ is an $\sfive$-algebra, $E_{\mathfrak B}$ is an equivalence relation (see~\cite[Thm.~3.18]{JT51}).
It remains to show that Definition~\ref{def:ms-frame}(E) is satisfied. Let $x,y,z \in X_{\mathfrak B}$ be such that $x E_{\mathfrak B} y$ and $y R_{\mathfrak B} z$. This means that $x \cap B_0=y \cap B_0$ and $y \cap H \subseteq z$. Let $F$ be the filter of $\mathfrak B$ generated by $(x \cap H) \cup (z \cap B_0)$. We show that $F$ is proper. Otherwise, since $x \cap H$ and $z \cap B_0$ are closed under meets, there are $a \in x \cap H$ and $b \in z \cap B_0$ such that $a \meet b =0$. Therefore, $a \le \neg b$. Thus, $a = \Box a \le \Box \neg b$, so $\Box \neg b \in x$. Since $B_0$ is an $\sfour$-subalgebra of $(B, \Box)$ and $b \in B_0$, we have $\Box \neg b \in B_0$. This yields $\Box \neg b \in x \cap B_0=y \cap B_0$, which implies $\Box \neg b \in y \cap H \subseteq z$. Therefore, $\neg b \in z$ which contradicts $b \in z$. Thus, $F$ is proper, and so there is an ultrafilter $u$ of $B$ such that $F\subseteq u$. Consequently, $x\cap H\subseteq u$ and $z\cap B_0\subseteq u\cap B_0$. Since $z \cap B_0$ and $u \cap B_0$ are both ultrafilters of $B_0$, we conclude that $z\cap B_0= u\cap B_0$. Thus, there is $u\in X_{\mathfrak B}$ with $xR_{\mathfrak B}u$ and $uE_{\mathfrak B}z$.
\end{proof}

\begin{definition}
We call an $\ms$-frame \textit{canonical} if it is isomorphic to $\mathfrak{B}_+$ for some $\ms$-algebra $\mathfrak{B}$.
\end{definition}

For an $\ms$-algebra $\mathfrak{B}$, it follows from \cite[Thm.~3.14]{JT51} that the Stone map $\beta:B \to \wp(X_{\mathfrak B})$ is a one-to-one homomorphism of $\ms$-algebras. Thus, we arrive at the following representation theorem.

\begin{proposition}
Each $\ms$-algebra $\mathfrak B$ is isomorphic to a subalgebra of $(\mathfrak B_+)^+$.
\end{proposition}

\begin{remark}
To recover the image of $\mathfrak B$ in $\wp(X_\mathfrak{B})$ we need to endow $X_\mathfrak{B}$ with a Stone topology. This leads to the notion of \emph{perfect $\ms$-frames} and a duality between the category of $\ms$-algebras and the category of perfect $\ms$-frames (see \cite{CarThesis} for details). When $\mathfrak{B}$ is finite, its embedding into $(\mathfrak B_+)^+$ is an isomorphism, and hence the categories of finite $\ms$-algebras and finite $\ms$-frames are dually equivalent.
\end{remark}

As an immediate consequence of the above considerations, we obtain that if $\mathfrak B$ is an $\ms$-algebra, then so is $(\mathfrak B_+)^+$. Thus, we have:

\begin{corollary}
$\ms$ is canonical.
\end{corollary}

A \emph{valuation} on an $\ms$-frame $\mathfrak F=(X,R,E)$ is a map $v$ associating a subset of $X$ to each propositional letter of
$\mathcal L_{\Box\forall}$. Then the boolean connectives are interpreted as usual,
\begin{equation*}
\begin{array}{l c l}
x \vDash_v \Box \varphi & \text{ iff } & (\forall y \in X) (x R y \, \Rightarrow \, y \vDash_v \varphi ), \\
x \vDash_v \forall \varphi & \text{ iff } & (\forall y \in X) (x E y \, \Rightarrow \, y \vDash_v \varphi ) .
\end{array}
\end{equation*}
As usual, we say that $\varphi$ is \emph{valid} in $\mathfrak F$, in symbols $\mathfrak F \vDash \varphi$, if $x \vDash_v \varphi$ for every
valuation $v$ and $x \in X$.

Soundness of $\ms$ with respect to this semantics is straightforward to prove, and completeness follows from the algebraic completeness and the representation theorem for $\ms$-algebras proved above.

\begin{theorem}\label{thm:relational semantics for ms4}
$\ms \vdash \varphi$ iff $\mathfrak F \vDash \varphi$ for every $\ms$-frame $\mathfrak F$.
\end{theorem}

In addition, $\ms$ has the fmp. While this can be proved directly using algebraic technique, we will derive it as a consequence of the fmp
of a stronger multimodal system in Section~\ref{sec:FMP}.

\subsection{G\"{o}del translation}

We recall that the G\"{o}del translation of $\mipc$ into $\ms$ is defined by
\begin{equation*}
\begin{array}{r c l l}
\bot^t &=& \bot & \\
p^t &=& \Box p & \text{for each propositional letter } p \\
(\varphi \land \psi)^t &=& \varphi^t \land \psi^t &\\
(\varphi \lor \psi)^t &=& \varphi^t \lor \psi^t &\\
(\varphi \to \psi)^t &=& \Box (\neg \varphi^t \lor \psi^t) &\\
(\forall \varphi)^t &=& \Box \forall \varphi^t & \\
(\exists \varphi)^t &=& \exists \varphi^t &
\end{array}
\end{equation*}

It was shown by Fischer-Servi \cite{FS77} that this translation is full and faithful, meaning that
\[
\mipc\vdash\varphi \mbox{ iff } \ms\vdash \varphi^t.
\]
Fischer-Servi used the translations of $\mipc$ and $\ms$ into
$\iqc$ and
$\qsfour$ respectively, and the predicate version of the G\"{o}del translation. In \cite{FS78a} she gave a different proof of this result using the fmp for $\mipc$. We give yet another proof utilizing relational semantics for $\mipc$ and $\ms$. Our proof generalizes the semantic proof that the G\"odel translation of $\ipc$ into $\sfour$ is full and faithful (see, e.g.,  \cite[Sec.~3.9]{CZ97}). We require the following lemma.

\begin{lemma}\label{lem:translation in ms4 is upset}
For any formula $\chi$ of $\mathcal L_{\forall\exists}$, we have
\[
\ms \vdash \chi^t \to \Box \chi^t.
\]
\end{lemma}

\begin{proof}
We first show that $\ms \vdash \exists \Box \varphi \to \Box \exists \varphi$ for any formula $\varphi$ of $\mathcal L_{\Box\forall}$. For this, by algebraic completeness, it is sufficient to prove that the inequality $\exists \Box a \leq \Box \exists a$ holds in every $\ms$-algebra $(B, \Box, \forall)$.
Let $a \in B$. We have
\[
\exists \Box a \leq \exists \Box \exists a = \exists \Box \forall \exists a \leq \exists \forall \Box \exists a = \forall \Box \exists a \leq \Box \exists a.
\]
We are now ready to prove that $\ms \vdash \chi^t \to \Box \chi^t$ by induction on the complexity of $\chi$. This is obvious when $\chi=\bot$. The cases when $\chi$ is $p$, $\varphi \to \psi$, or $\forall \varphi$ follow from the axiom $\Box \varphi \to \Box \Box \varphi$. We next consider the cases when $\chi$ is $\varphi \wedge \psi$ or $\varphi \vee \psi$. Suppose that the claim is true for $\varphi$ and $\psi$, so $\varphi^t \to \Box \varphi^t$ and $\psi^t \to \Box \psi^t$ are theorems of $\ms$. Then $\varphi^t \wedge \psi^t \to \Box (\varphi^t \wedge \psi^t)$ and $\varphi^t \vee \psi^t \to \Box (\varphi^t \vee \psi^t)$ are also theorems of $\ms$. Finally, if $\chi$ is $\exists \varphi$ and $\ms \vdash \varphi^t\to \Box \varphi^t$, then $\ms \vdash \exists \varphi^t \to \exists \Box \varphi^t$. Therefore, since $\ms \vdash \exists \Box \varphi^t \to \Box \exists \varphi^t$, we conclude that $\ms \vdash \exists \varphi^t \to \Box \exists \varphi^t$.
\end{proof}

In the next definition we generalize to $\ms$-frames the well-known definition of skeleton (see, e.g., \cite[Sec.~3.9]{CZ97}).

\begin{definition}\label{def:skeleton of ms4-frames}
Let $\mathfrak F=(X,R,E)$ be an $\ms$-frame. Define the relation
$Q_E$ on $X$ by setting $x Q_E y$ iff $(\exists  z \in X )(xRz \ \& \ zEy)$. Then the \textit{skeleton} $\mathfrak{F}^t=(X',R',Q')$ of $\mathfrak{F}$ is defined as follows. Let $\sim$ be the equivalence relation on $X$ given by $x \sim y$ iff $xRy$ and $yRx$. We let $X'$ be the set of equivalence classes of $\sim$, and define $R'$ and $Q'$ on $X'$ by $[x] R' [y]$ iff $xRy$ and $[x] Q' [y]$ iff $xQ_Ey$.
\end{definition}

\begin{proposition} \label{prop:skeleton of ms4-frames}
\begin{enumerate}
\item[]
\item If $\mathfrak F$ is an $\ms$-frame, then $\mathfrak{F}^t$ is an $\mipc$-frame.
\item For each valuation $v$ on $\mathfrak{F}$ there is a valuation $v'$ on $\mathfrak{F}^t$ such that for each $x\in \mathfrak F$ and $\mathcal{L}_{\forall \exists}$-formula $\varphi$, we have
\[
\mathfrak{F}^t, [x] \vDash_{v'} \varphi \mbox{ iff } \mathfrak{F}, x \vDash_v \varphi^t.
\]
\item For each $\mathcal{L}_{\forall \exists}$-formula $\varphi$, we have
\[
\mathfrak{F}^t \vDash \varphi \mbox{ iff } \mathfrak{F} \vDash \varphi^t.
\]
\item For each $\mipc$-frame $\mathfrak{G}$ there is an $\ms$-frame $\mathfrak{F}$ such that $\mathfrak{G}$ is isomorphic to $\mathfrak{F}^t$.
\end{enumerate}
\end{proposition}

\begin{proof}
(1). It is well known that $(X',R')$ is an intuitionistic Kripke frame. That $Q'$ is well defined follows from Condition (E). Showing that $Q'$ is a quasi-order, and that (O1) and (O2) hold in $\mathfrak{F}^t$ is straightforward.

(2). Define $v'$ on $\mathfrak{F}^t$ by $v'(p)=\{ [x] \in X' \mid R[x] \subseteq v(p)\}$. We show that $\mathfrak{F}^t, [x] \vDash_{v'} \varphi$ iff $\mathfrak{F},x \vDash_v \varphi^t$ by induction on the complexity of $\varphi$. Since $v'(p)= \{ [x] \mid \mathfrak{F},x \vDash_v \Box p \}$, the claim is obvious when $\varphi$ is a propositional letter.
We prove the claim for $\varphi$ of the form $\forall \psi$ and $\exists \psi$ since the other cases are well known. Suppose $\varphi=\forall \psi$. By the definition of $Q'$ and induction hypothesis, we have
\begin{align*}
\mathfrak{F}^t, [x] \vDash_{v'} \forall \psi &\mbox{ iff } (\forall [y] \in X')([x]Q'[y] \, \Rightarrow \, \mathfrak{F}^t, [y] \vDash_{v'} \psi)\\
 &\mbox{ iff } (\forall y \in X)(xQ_Ey \, \Rightarrow \, \mathfrak{F}^t, [y] \vDash_{v'} \psi)\\
 &\mbox{ iff } (\forall y \in X)(xQ_Ey \, \Rightarrow \, \mathfrak{F}, y \vDash_v \psi^t).
\end{align*}
On the other hand,
\begin{align*}
\mathfrak{F}, x \vDash_v (\forall \psi)^t &\mbox{ iff } \mathfrak{F}, x \vDash_v \Box \forall \psi^t\\
 &\mbox{ iff } (\forall z \in X)(xRz \, \Rightarrow \, (\forall y \in X)(zEy \, \Rightarrow \, \mathfrak{F}, y \vDash_v \psi^t))\\
 &\mbox{ iff } (\forall y \in X)(xQ_Ey \, \Rightarrow \, \mathfrak{F}, y \vDash_v \psi^t).
\end{align*}
Thus, $\mathfrak{F}^t, [x] \vDash_{v'} \forall \psi$ iff $\mathfrak{F}, x \vDash_v (\forall \psi)^t$.

Suppose $\varphi=\exists \psi$. As noted in Remark~\ref{rem:E_Q and Q on R-upsets}, $Q'$ and $E_{Q'}$ coincide on $R'$-upsets, and it is straightforward to see by induction that the set $\{ [y] \mid \mathfrak{F}^t, [y] \vDash_{v'} \psi \}$ is an $R'$-upset. Therefore, by the induction hypothesis,
\begin{align*}
\mathfrak{F}^t, [x] \vDash_{v'} \exists \psi &\mbox{ iff } (\exists [y] \in X')([x]E_{Q'}[y] \; \& \; \mathfrak{F}^t, [y] \vDash_{v'} \psi)\\
&\mbox{ iff } [x] \in E_{Q'}[\{ [y] \mid \mathfrak{F}^t, [y] \vDash_{v'} \psi \}]\\
&\mbox{ iff } [x] \in Q'[\{ [y] \mid \mathfrak{F}^t, [y] \vDash_{v'} \psi \}]\\
 &\mbox{ iff } x \in Q_E[\{ y \mid \mathfrak{F}^t, [y] \vDash_{v'} \psi \}]\\
  &\mbox{ iff } x \in Q_E[\{ y \mid \mathfrak{F}, y \vDash_v \psi^t \}].
\end{align*}
On the other hand,
\begin{align*}
\mathfrak{F}, x \vDash_v (\exists \psi)^t &\mbox{ iff } \mathfrak{F}, x \vDash_v \exists \psi^t \\
 &\mbox{ iff } (\exists y \in X)(xEy \; \& \; \mathfrak{F}, y \vDash_v \psi^t)\\
 &\mbox{ iff } x \in E[\{ y \mid \mathfrak{F}, y \vDash_v \psi^t \}]\\
 &\mbox{ iff } x \in Q_E[\{ y \mid \mathfrak{F}, y \vDash_v \psi^t \}]
\end{align*}
since, by Lemma~\ref{lem:translation in ms4 is upset}, the set $\{ y \mid \mathfrak{F}, y \vDash_v \psi^t \}$ is an $R$-upset, and
$E$ and $Q_E$ coincide on $R$-upsets. Thus, $\mathfrak{F}^t, [x] \vDash_{v'} \exists \psi$ iff $\mathfrak{F}, x \vDash_v (\exists \psi)^t$.

(3). If $\mathfrak{F} \nvDash \varphi^t$, then there is a valuation $v$ on $\mathfrak{F}$ such that $\mathfrak{F}, x \nvDash_v \varphi^t$ for some $x \in X$. By (2), $v'$ is a valuation on $\mathfrak{F}^t$ such that $\mathfrak{F}^t, [x] \nvDash_{v'} \varphi$. Therefore, $\mathfrak{F}^t \nvDash \varphi$. If $\mathfrak{F}^t \nvDash \varphi$, then there is a valuation $w$ on $\mathfrak{F}^t$ and $[x] \in X'$ such that $\mathfrak{F}^t, [x] \nvDash_w \varphi$. Let $v$ be the valuation on $\mathfrak{F}$ given by $v(p)=\{x \mid [x] \in w(p)\}$. Since $\mathfrak{F}^t$ is an $\mipc$-frame, $w(p)$ is an $R'$-upset of $\mathfrak{F}^t$ for each $p$. So $v(p)$ is an $R$-upset of $\mathfrak{F}$ for each $p$. Therefore, $w=v'$ because
\[
v'(p)=\{ [x] \in X' \mid R[x] \subseteq v(p) \}=\{ [x] \in X' \mid x \in v(p) \}=w(p).
\]
Thus, $\mathfrak{F}^t, [x] \nvDash_{v'} \varphi$. By (2), $\mathfrak{F}, x \nvDash_v \varphi^t$. Consequently, $\mathfrak{F} \nvDash \varphi^t$.

(4). Let $\mathfrak{G}=(X,R,Q)$ be an $\mipc$-frame. We show that  $\mathfrak{F}=(X,R,E_Q)$ is an $\ms$-frame. If $xE_Qy$ and $yRz$, then by definition of $E_Q$ and condition (O1) of $\mipc$-frames, $xQy$ and $yQz$. Since $Q$ is transitive, $xQz$. Condition (O2) then implies that there is $u \in X$ with $xRu$ and $u E_Q z$. Thus, $\mathfrak F$ is an $\ms$-frame. Since $R$ is a partial order, $\sim$ is the identity relation. It then follows from condition (O2) that $Q=Q_{E_Q}$, and hence $\mathfrak{G}$ is isomorphic to $\mathfrak{F}^t$.
\end{proof}

\begin{remark}
In general, we cannot recover an $\ms$-frame $\mathfrak{F}=(X,R,E)$ from its skeleton $\mathfrak{F}^t$ even if $R$ is a partial order. Indeed, it is not always the case that $E=E_{Q_E}$.
However, if $\mathfrak{F}$ is canonical (and in particular finite), then $E=E_{Q_E}$; see~\cite[Sec.~2]{Bez99} for details.
\end{remark}

\begin{theorem} \label{thm:godel translation}
The G\"{o}del translation of $\mipc$ into $\ms$ is full and faithful; that is, \begin{equation*}
\mipc \vdash\varphi \quad \mbox{ iff } \quad \ms \vdash \varphi^t.
\end{equation*}
\end{theorem}

\begin{proof}
To prove faithfulness, suppose that $\ms \nvdash\varphi^t$. By Theorem~\ref{thm:relational semantics for ms4}, there is an $\ms$-frame $\mathfrak{F}$ such that $\mathfrak{F} \nvDash \varphi^t$. By Proposition~\ref{prop:skeleton of ms4-frames}, $\mathfrak{F}^t$ is an $\mipc$-frame and $\mathfrak{F}^t \nvDash \varphi$. Thus, by Theorem~\ref{thm:relational semantics for mipc}, $\mipc \nvdash \varphi$. For fullness, let $\mipc \nvdash \varphi$. Then there is an $\mipc$ frame $\mathfrak{G}$ such that $\mathfrak{G} \nvDash \varphi$. By Proposition~\ref{prop:skeleton of ms4-frames}(4), there is an $\ms$-frame such that $\mathfrak{G}$ isomorphic to $\mathfrak{F}^t$. Therefore, $\mathfrak{F}^t \nvDash \varphi$. Proposition~\ref{prop:skeleton of ms4-frames}(3) implies that $\mathfrak{F} \nvDash \varphi^t$. Thus, $\ms \nvdash \varphi^t$.
\end{proof}

\begin{remark}
The original proof of McKinsey and Tarski \cite{MT46, MT48} that the G\"odel translation of $\sf IPC$ into $\sf S4$ is full and faithful was algebraic. They proved that the $\Box$-fixpoints of each $\sf S4$-algebra form a Heyting algebra, and that each Heyting algebra arises this way.
In the monadic setting, while we still have that the $\Box$-fixpoints of each $\ms$-algebra form a monadic Heyting algebra, it is not the case that each monadic Heyting algebra arises this way (see~\cite{CarThesis} for details).
Nevertheless, Fischer-Servi~\cite{FS78a} proved that each finite monadic Heyting algebra does.
Thus, while we can prove faithfulness in the same fashion as McKinsey and Tarski, proving fullness requires to first establish the finite model property for $\mipc$.
\end{remark}

\section{Translation of $\mipc$ into $\ts$}\label{sec:TS4}

In this section we introduce a new multimodal tense system $\ts$ in which, as we will show in the next section, $\mipc$ embeds fully and faithfully by a modified G\"odel translation. For this we require to recall the well-known tense system $\st$.

\subsection{$\st$}

The tense logic $\st$ is the extension of the least tense logic $\sf K.t$ in which both tense modalities satisfy the $\sfour$-axioms. This system was studied by several authors. In particular, Esakia \cite{Esa75} showed that an extension of the G\"odel translation embeds the Heyting-Brouwer logic $\sf HB$ of Rauszer \cite{Rau73} into $\st$ fully and faithfully. The language of $\sf HB$ is obtained by enriching the language of $\sf IPC$ by an additional connective of coimplication, and the logic $\sf HB$ is the extension of $\sf IPC$ by the axioms for coimplication, which are dual to the axioms for implication. Wolter \cite{Wol98} extended the celebrated Blok-Esakia Theorem to this setting.

Let $\mathcal L_T$ be the propositional tense language with two modalities $\boxf$ and $\boxp$. As usual, $\boxf$ is interpreted as ``always in the future" and $\boxp$ as ``always in the past." We use the following standard
abbreviations: $\diaf$ for $\neg \boxf \neg$ and $\diap$ for $\neg \boxp \neg$. Then $\diaf$ is interpreted as ``sometime in the future" and $\diap$ as ``sometime in the past."

\begin{definition}
Let $\st$ be the smallest classical bimodal logic containing the $\sfour$-axioms for $\boxf$ and $\boxp$, the tense axioms
\begin{equation*}
\begin{array}{c}
p \to \boxp \diaf p \\
p \to \boxf \diap p
\end{array}
\end{equation*}
and closed under modus ponens, substitution, $\boxf$-necessitation, and $\boxp$-necessitation.
\end{definition}

Algebraic semantics for $\st$ was studied by Esakia \cite{Esa75,Esa78}, where the duality theory for $\sf S4$-algebras was generalized to $\st$-algebras.

\begin{definition}\label{def:ts-algebra}
An $\st$-algebra is a triple $(B, \boxf, \boxp)$ where $(B,\boxf)$, $(B,\boxp)$ are $\sfour$-algebras and for each $a\in B$ we have
\begin{equation}\label{PF}
\tag{PF}
a \le \boxp \diaf a
\end{equation}
\begin{equation}\label{FP}
\tag{FP}
a \le \boxf \diap a
\end{equation}
\end{definition}

The Lindenbaum-Tarski construction yields that $\st$-algebras provide a sound and complete algebraic semantics for $\st$. Relational semantics for $\st$ is given by $\st$-frames.

\begin{definition}
An $\st$-frame is a pair $\mathfrak{F}=(X,Q)$ where $X$ is a set and $Q$ is a quasi-order on $X$.
\end{definition}

Let $\qinv$ be the converse of $Q$. For $U\in\wp(X)$ let
\[
\boxq(U)=X\setminus Q^{-1}[X\setminus U] \mbox{ and } \boxqinv(U)=X\setminus Q[X\setminus U].
\]
Since $Q$ is a quasi-order, so is $\qinv$, so $(\wp(X),\boxq)$ and $(\wp(X),\boxqinv)$ are $\sfour$-algebras. A standard argument (see~\cite[Thm.~3.6]{JT51}) gives that $\mathfrak F^+:=(\wp(X),\boxq, \boxqinv)$ satisfies (\ref{PF}) and (\ref{FP}). Therefore, $\mathfrak F^+$ is an $\st$-algebra.

In fact, each $\st$-algebra $\mathfrak B=(B,\boxf,\boxp)$ is isomorphic to a subalgebra of $\mathfrak F^+$ for some $\st$-frame $\mathfrak F$. As usual, we can take $\mathfrak F$ to be the canonical frame of $\mathfrak B$.
Let $H_F$ and $H_P$ be the sets of $\boxf$-fixpoints and $\boxp$-fixpoints, respectively. Since $\boxf$ and $\boxp$ are $\sfour$-operators, $H_F$ and $H_P$ are Heyting algebras.

\begin{remark}\label{rem:dual iso H_F and H_P}
Let $(B, \boxf, \boxp)$ be an $\st$-algebra. It follows from Definition~\ref{def:ts-algebra} that $H_F$ coincides with the set of $\diap$-fixpoints and $H_P$ with the set of $\diaf$-fixpoints. Moreover, $\neg$ maps $H_F$ to $H_P$ and vice versa. Indeed, if $a \in H_F$, then $a = \boxf a$. By~(\ref{PF}), $\diap a = \diap \boxf a \le a$, so $\diap a = a$, and hence  $\boxp \neg a =\neg \diap a= \neg a$. Therefore, $\neg a \in H_P$. Similarly, if $a \in H_P$, then $\neg a \in H_F$. Thus, $\neg$ is a dual isomorphism between $H_F$ and $H_P$.
\end{remark}

\begin{definition}\label{def:canonical frame for st}
Let $\mathfrak B=(B,\boxf,\boxp)$ be an $\st$-algebra. The \emph{canonical frame} of $\mathfrak B$ is the frame
$\mathfrak B_+=(X_{\mathfrak B},Q_{\mathfrak B})$ where
$X_{\mathfrak B}$ is the set of ultrafilters of $B$ and $x Q_{\mathfrak B} y$ iff $x \cap H_F \subseteq y$; equivalently, $y\cap H_P\subseteq x$.
\end{definition}

By a standard argument, if $\mathfrak B$ is an $\st$-algebra, then $\mathfrak B_+$ is an $\st$-frame.

\begin{definition}
We call an $\st$-frame \textit{canonical} if it is isomorphic to $\mathfrak{B}_+$ for some $\st$-algebra $\mathfrak{B}$.
\end{definition}

A standard argument now yields the following representation theorem.

\begin{proposition} \label{prop:representation st}
If $\mathfrak B$ is an $\st$-algebra, then $\mathfrak B$ is isomorphic to a subalgebra of $(\mathfrak B_+)^+$.
\end{proposition}

\begin{remark}
To recover the image of $\mathfrak B$ in $\wp(X_\mathfrak{B})$ we need to endow $X_\mathfrak{B}$ with a Stone topology. This leads to the notion of perfect $\st$-frames and a duality between the category of $\st$-algebras and the category of perfect $\st$-frames (see \cite{Esa78}). When $\mathfrak{B}$ is finite, its embedding into $(\mathfrak B_+)^+$ is an isomorphism, and hence the categories of finite $\st$-algebras and finite $\st$-frames are dually equivalent.
\end{remark}

As an immediate consequence, we obtain:

\begin{corollary}
$\st$ is canonical.
\end{corollary}

While $\st$-frames coincide with $\sfour$-frames, the difference is in the interpretation of the modalities as we use $Q$ to interpret $\boxf$ and $\qinv$ to interpret $\boxp$.

A \emph{valuation} on an $\st$-frame $\mathfrak F=(X,Q)$ is a map $v$ associating a subset of $X$ to each propositional letter of
$\mathcal L_T$. The classical connectives are interpreted as usual, and the tense modalities are interpreted as
\begin{equation*}
\begin{array}{l c l}
x \vDash_v \boxf \varphi & \text{ iff } & (\forall y \in X) (x Q y \, \Rightarrow \, y \vDash_v \varphi ), \\
x \vDash_v \boxp \varphi & \text{ iff } & (\forall y \in X) (y Q x \, \Rightarrow \, y \vDash_v \varphi ).
\end{array}
\end{equation*}
As usual, we say that $\varphi$ is \emph{valid} in $\mathfrak F$, in symbols $\mathfrak F \vDash \varphi$, if $x \vDash_v \varphi$ for every
valuation $v$ and $x \in X$.

Soundness of $\st$ with respect to this semantics is straightforward to prove. Completeness follows from the algebraic completeness and the representation of $\st$-algebras.

\begin{theorem}\label{thm:relational semantics for st}
$\st \vdash \varphi$ iff $\mathfrak F \vDash \varphi$ for every $\st$-frame $\mathfrak F$.
\end{theorem}

That $\st$ has the fmp belongs to folklore. We were unable to find it stated explicitly in the literature. It will follow from our results in Section~\ref{sec:FMP}.

\subsection{$\ts$}

The tense logic $\ts$ will combine $\sfour$ with $\st$. We will use $\sfour$ to interpret intuitionistic connectives, and $\st$ to interpret monadic intuitionistic quantifiers. Let $\mathcal{ML}$ be the multimodal propositional language with three modalities $\Box$, $\boxf$, and $\boxp$. We use $\Diamond$, $\diaf$, and $\diap$ as usual abbreviations.

\begin{definition}
The logic $\ts$ is the least classical multimodal logic containing the $\sfour$-axioms for $\Box$, $\boxf$, and $\boxp$, the tense axioms for $\boxf$ and $\boxp$, the connecting axioms
\begin{equation*}
\begin{array}{l}
\Diamond p \to \diaf p \\
\diaf p \to \Diamond( \diaf p \meet \diap p)
\end{array}
\end{equation*}
and closed under modus ponens, substitution, and three necessitation rules (for $\Box$, $\boxf$, and $\boxp$).
\end{definition}

Algebraic semantics for $\ts$ is given by $\ts$-algebras.

\begin{definition} \label{ts4algdef}
A $\ts$-algebra is a quadruple $\mathfrak B=(B, \Box, \boxf, \boxp)$ where $(B, \Box)$ is an $\sfour$-algebra, $(B, \boxf, \boxp)$ is an $\st$-algebra, and for each $a \in B$ we have:
\begin{equation}\label{T1}
\tag{T1}
\Diamond a \le \diaf a
\end{equation}
\begin{equation}\label{T2}
\tag{T2}
\diaf a \le \Diamond( \diaf a \meet \diap a)
\end{equation}
\end{definition}

The Lindenbaum-Tarski construction then yields that $\ts$-algebras provide a sound and complete algebraic semantic for $\ts$.

\begin{definition}
A $\ts$-frame is a triple $\mathfrak F=(X,R,Q)$ where $X$ is a set and $R,Q$ are quasi-orders on $X$ such that $R \subseteq Q$ and $x Q y$ implies that there is $z \in X$ such that $x R z$ and $z E_Q y$.
\end{definition}

\begin{remark}\label{rem:correspondence mipc, ms and ts frames}
\begin{enumerate}
\item[]
\item The only difference between $\ts$-frames and $\mipc$-frames is that in $\ts$-frames the relation $R$ is a quasi-order, while in $\mipc$-frames it is a partial order.
\item It is straightforward to check that if $(X,R,Q)$ is a $\ts$-frame, then $(X,R,E_Q)$ is an $\ms$-frame, and that if $(X,R,E)$ is an $\ms$-frame, then $(X,R,Q_E)$ is a $\ts$-frame. (We recall that, as in Definition~\ref{def:skeleton of ms4-frames}, $Q_E$ is defined by $x Q_E y$ iff $(\exists  z \in X )(xRz \ \& \ zEy)$). 
If $(X,R,Q)$ is a $\ts$-frame, by definition we have that $x Q y $ iff $(\exists z \in X) (x R z \ \& \ z E_Q y)$. Thus, $Q=Q_{E_Q}$. On the other hand, there exist $\ms$-frames $(X,R,E)$ such that $E\ne E_{Q_E}$ (see \cite[p.~24]{Bez99}). Therefore, this correspondence is
not a bijection.
\end{enumerate}
\end{remark}

\begin{lemma} \label{prop: X^+ for ts4}
If $\mathfrak F=(X,R,Q)$ is a $\ts$-frame, then $\mathfrak F^+=(\wp(X), \boxr, \boxq, \boxqinv)$ is a $\ts$-algebra.
\end{lemma}

\begin{proof}
Since $R$ and $Q$ are quasi-orders, $(\wp(X), \boxr)$ is an $\sfour$-algebra and $(\wp(X), \boxq, \boxqinv)$ is an $\st$-algebra. It remains to show that $\mathfrak{F}^+$ satisfies (\ref{T1}) and (\ref{T2}).
\begin{itemize}
\item[(T1)] Since $R \subseteq Q$, we have $\diar(U) = R^{-1}[U] \subseteq Q^{-1}[U] = \diaq(U)$.
\item[(T2)] Let $x \in \diaq(U) = Q^{-1}[U]$, so there is $y \in U$ with $x Q y$. Then there is $z \in X$ with $x R z$ and $z E_Q y$. Therefore, $z \in Q^{-1}[y] \subseteq Q^{-1}[U] = \diaq(U)$ and $z \in Q[y] \subseteq Q[U]= \diaqinv(U)$. Thus,
$x \in R^{-1}[z] \subseteq R^{-1} [\diaq(U) \cap \diaqinv(U)] = \diar (\diaq(U) \cap \diaqinv(U))$. This shows that $\diaq(U) \subseteq \diar (\diaq(U) \cap \diaqinv(U))$.
\end{itemize}
\end{proof}

We next prove that each $\ts$-algebra is represented as a subalgebra of $\mathfrak F^+$ for some $\ts$-frame $\mathfrak F$. For a $\ts$-algebra $(B, \Box, \boxf, \boxp)$ let $H$, $H_F$, and $H_P$ be the Heyting algebras of the $\Box$-fixpoints, $\boxf$-fixpoints, and $\boxp$-fixpoints, respectively.

\begin{definition}\label{def:canonical frame ts4}
Let $\mathfrak B=(B, \Box, \boxf, \boxp)$ be a $\ts$-algebra. The \emph{canonical frame} of $\mathfrak B$ is the frame
$\mathfrak B_+=(X_{\mathfrak B},R_{\mathfrak B},Q_{\mathfrak B})$ where
$X_{\mathfrak B}$ is the set of ultrafilters of $B$, $x R_{\mathfrak B} y$ iff $x\cap H \subseteq y$, and $x Q_{\mathfrak B} y$ iff $x\cap H_F \subseteq y$, which happens iff $y\cap H_P\subseteq x$.
\end{definition}

\begin{lemma}
If $\mathfrak B$ is a $\ts$-algebra, then $\mathfrak B_+$ is a $\ts$-frame.
\end{lemma}

\begin{proof}
Clearly $R_\mathfrak{B}$ and $Q_\mathfrak{B}$ are quasi-orders. To prove that $R_\mathfrak{B} \subseteq Q_\mathfrak{B}$ we first show that $H_F \subseteq H$. Let $a \in H_F$. Then $a = \boxf a =\neg \diaf \neg a =\neg \diaf \diaf \neg a$. By (\ref{T1}),
\[
\neg \diaf \diaf \neg a \le \neg \Diamond \diaf \neg a = \Box \boxf a \le \Box a.
\]
Therefore, $a=\Box a$, and so $a\in H$. Now suppose that $x R_\mathfrak{B} y$, so $x \cap H \subseteq y$. Let $a\in x \cap H_F$. Then $a\in x \cap H\subseteq y$. Thus, $a\in y$, and hence $x Q_\mathfrak{B} y$.

To prove the other condition, let $x Q_\mathfrak{B} y$, so $x\cap H_F\subseteq y$. We show that $(x \cap H) \cup (y \cap H_F) \cup (y \cap H_P)$ generates a proper filter of $B$. Otherwise, since $H, H_F, H_P$ are closed under meets, there are $a \in x \cap H$, $b \in y \cap H_F$, and $c \in y \cap H_P$ such that $a \meet b \meet c=0$. By Remark~\ref{rem:dual iso H_F and H_P}, $H_F$ coincides with the set of $\diap$-fixpoints and $H_P$ with the set of $\diaf$-fixpoints. Therefore, since $b \in H_F$ and $c \in H_P$, we have $\diap (b\meet c) \meet \diaf (b\meet c) \le \diap b \meet \diaf c =b \meet c$. Thus, $a \meet \diap (b\meet c) \meet \diaf (b\meet c) \le a \meet b \meet c=0$, yielding $a \le \neg(\diap (b\meet c) \meet \diaf (b\meet c))$. Since $a \in H$, we have
\[
a=\Box a \le \Box \neg(\diap (b\meet c) \meet \diaf (b\meet c))=\neg \Diamond (\diap (b\meet c) \meet \diaf (b\meet c)).
\]
Consequently, $a \meet \Diamond (\diap (b\meet c) \meet \diaf (b\meet c))=0$. By (\ref{T2}),
\[
a \meet \diaf (b\meet c) \le a \meet \Diamond (\diap (b\meet c) \meet \diaf (b\meet c))=0.
\]
Because $b\meet c \le \diaf (b\meet c)$, $b \meet c \in y$, and $y$ is a filter, we have $\diaf (b\meet c) \in y$. Since $x\cap H_F\subseteq y$, we have $y\cap H_P\subseteq x$. Therefore, $\diaf (b\meet c) \in y \cap H_P \subseteq x$ and $a \in x$. Thus, $0=a \meet \diaf (b\meet c) \in x$, a contradiction.
Consequently, there is an ultrafilter $z$ such that $(x \cap H) \cup (y \cap H_F) \cup (y \cap H_P) \subseteq z$. But then $x \cap H \subseteq z$, $y \cap H_F \subseteq z$, and $y \cap H_P \subseteq z$. This gives that $x R_\mathfrak{B} z$, $y Q_\mathfrak{B} z$, and $z Q_\mathfrak{B} y$, as desired.
\end{proof}

\begin{definition}
We call a $\ts$-frame \emph{canonical} if it is isomorphic to $\mathfrak{B}_+$ for some $\ts$-algebra $\mathfrak{B}$.
\end{definition}

Let $\mathfrak B$ be a $\ts$-algebra. Since $\beta:B \to \wp(X_{\mathfrak B})$ is an embedding of $\ts$-algebras, we obtain the following representation theorem for $\ts$-algebras.

\begin{proposition} \label{prop: representation for ts4}
Each $\ts$-algebra $\mathfrak B$ is isomorphic to a subalgebra of $(\mathfrak B_+)^+$.
\end{proposition}

\begin{remark}
To recover the image of $\mathfrak B$ in $\wp(X_\mathfrak{B})$ we need to endow $X_\mathfrak{B}$ with a Stone topology. This leads to the notion of perfect $\ts$-frames and a duality between the categories of $\ts$-algebras and perfect $\ts$-frames (see \cite{CarThesis} for details). When $\mathfrak{B}$ is finite, its embedding into $(\mathfrak B_+)^+$ is an isomorphism, and hence the categories of finite $\ts$-algebras and finite $\ts$-frames are dually equivalent.
\end{remark}

Since $(\mathfrak B_+)^+$ is a $\ts$-algebra, as an immediate consequence we obtain:

\begin{corollary}
$\ts$ is canonical.
\end{corollary}

Let $\mathfrak F=(X,R,Q)$ be a $\ts$-frame. A \emph{valuation} of $\mathcal{ML}$ into $\mathfrak F$ associates with each propositional letter
a subset of $X$.
The classical connectives are interpreted as usual, $\Box$ is interpreted using the relation $R$, and $\boxf$, $\boxp$ are interpreted using
the relation $Q$:
\begin{equation*}
\begin{array}{l c l}
x \vDash_v \Box \varphi & \text{ iff } & (\forall y \in X) (x R y \Rightarrow y \vDash_v \varphi ), \\
x \vDash_v \boxf \varphi & \text{ iff } & (\forall y \in X) (x Q y \Rightarrow y \vDash_v \varphi ), \\
x \vDash_v \boxp \varphi & \text{ iff } & (\forall y \in X) (y Q x \Rightarrow y \vDash_v \varphi ).
\end{array}
\end{equation*}
Consequently,
\begin{equation*}
\begin{array}{l c l}
x \vDash_v \Diamond \varphi & \text{ iff } & (\exists y \in X) (x R y \; \& \; y \vDash_v \varphi ),\\
x \vDash_v \diaf \varphi & \text{ iff } & (\exists y \in X) (x Q y \; \& \; y \vDash_v \varphi ),\\
x \vDash_v \diap \varphi & \text{ iff } & (\exists y \in X) (y Q x \; \& \; y \vDash_v \varphi ).
\end{array}
\end{equation*}

\begin{theorem}\label{thm:relational semantics for ts4}
$\ts \vdash \varphi$ iff $\mathfrak F \vDash \varphi$ for every $\ts$-frame $\mathfrak F$.
\end{theorem}

\begin{proof}
Soundness is straightforward to prove, and completeness follows from the algebraic completeness and the representation of $\ts$-algebras (Proposition~\ref{prop: representation for ts4}).
\end{proof}

In Section~\ref{sec:FMP} we will prove that $\ts$ has the fmp.

\subsection{G\"{o}del translation adjusted}

In this section we modify the G\"{o}del translation to embed $\mipc$ into $\ts$ fully and faithfully.

\begin{definition}
The translation $(-)^\natural: \mipc \to \ts$ is defined as $(-)^t$ on propositional letters, $\bot$, $\wedge$, $\vee$, and $\to$; and for $\forall$ and $\exists$ we set:
\begin{equation*}
\begin{array}{c}
(\forall \varphi)^\natural= \boxf \varphi^\natural\\
(\exists \varphi)^\natural= \diap \varphi^\natural
\end{array}
\end{equation*}
Thus, $\forall$ is interpreted as ``always in the future'' and $\exists$ as ``sometime in the past.''
\end{definition}

We adapt Definition~\ref{def:skeleton of ms4-frames} to the setting of $\ts$-frames by utilizing the correspondence between $\ts$-frames and $\ms$-frames described in Remark~\ref{rem:correspondence mipc, ms and ts frames}.

\begin{definition}\label{def:skeleton of ts4-frames}
Let $\mathfrak F=(X,R,Q)$ be a $\ts$-frame, and let $\sim$ be the equivalence relation given by $x \sim y$ iff $xRy$ and $yRx$. We set $X'$ to be the set of equivalence classes of $\sim$, and define $R'$ and $Q'$ on $X'$ by $[x] R' [y]$ iff $xRy$ and $[x] Q' [y]$ iff $xQy$. We call  $\mathfrak{F}^\natural=(X',R',Q')$ the \textit{skeleton} of $\mathfrak{F}$.
\end{definition}

\begin{proposition} \label{prop:skeleton of ts4-frames}
\begin{enumerate}
\item[]
\item If $\mathfrak F$ is a $\ts$-frame, then $\mathfrak{F}^\natural$ is an $\mipc$-frame.
\item For each valuation $v$ on $\mathfrak{F}$ there is a valuation $v'$ on $\mathfrak{F}^\natural$ such that for each $x \in \mathfrak{F}$ and $\mathcal{L}_{\forall \exists}$-formula $\varphi$, we have
\[
\mathfrak{F^\natural}, [x] \vDash_{v'} \varphi \mbox{ iff } \mathfrak{F},x \vDash_{v} \varphi^\natural.
\]
\item For each $\mathcal{L}_{\forall \exists}$-formula $\varphi$, we have
\[
\mathfrak{F}^\natural \vDash \varphi \mbox{ iff } \mathfrak{F} \vDash \varphi^\natural.
\]
\item Any $\mipc$-frame $\mathfrak{G}$ is also a $\ts$-frame and $\mathfrak{G}^\natural$ is isomorphic to $\mathfrak{G}$.
\end{enumerate}
\end{proposition}

\begin{proof}
(1). It is well known that $(X',R')$ is an intuitionistic Kripke frame. $Q'$ is well defined on $X'$ because $R \subseteq Q$ in $\mathfrak{F}$. Showing that $Q'$ is a quasi-order, and that (O1) and (O2) hold in $\mathfrak{F}^\natural$ is straightforward.

(2).
As in Proposition~\ref{prop:skeleton of ms4-frames}(2), we define $v'$ by $v'(p)=\{ [x] \in X' \mid R[x] \subseteq v(p)\}$,
and show that $\mathfrak{F^\natural}, [x] \vDash_{v'} \varphi$ iff $\mathfrak{F},x \vDash_{v} \varphi^\natural$ by induction on the complexity of $\varphi$.
It is sufficient to only consider the cases when $\varphi$ is of the form $\forall \psi$ or $\exists \psi$.
Suppose $\varphi=\forall \psi$. Then by the definition of $Q'$ and induction hypothesis,
\begin{align*}
\mathfrak{F}^\natural, [x] \vDash_{v'} \forall \psi &\mbox{ iff } (\forall [y] \in X')([x]Q'[y] \, \Rightarrow \, \mathfrak{F}^\natural, [y] \vDash_{v'} \psi)\\
 &\mbox{ iff } (\forall y \in X)(xQy \, \Rightarrow \, \mathfrak{F}^\natural, [y] \vDash_{v'} \psi) \\
 &\mbox{ iff } (\forall y \in X)(xQy \, \Rightarrow \, \mathfrak{F}, y \vDash_v \psi^\natural) \\
  &\mbox{ iff } \mathfrak{F}, x \vDash_v \boxf \psi^\natural \\
   &\mbox{ iff }  \mathfrak{F}, x \vDash_v (\forall \psi)^\natural.
\end{align*}
Suppose $\varphi=\exists \psi$. As noted in Remark~\ref{rem:E_Q and Q on R-upsets}, $Q'$ and $E_{Q'}$ coincide on $R'$-upsets. Since the set $\{ [y] \mid \mathfrak{F}^\natural, [y] \vDash_{v'} \psi \}$ is an $R'$-upset,
by the induction hypothesis, we have
\begin{align*}
\mathfrak{F}^\natural, [x] \vDash_{v'} \exists \psi &\mbox{ iff } (\exists [y] \in X')([x]E_{Q'}[y] \; \& \; \mathfrak{F}^\natural, [y] \vDash_{v'} \psi)\\
 &\mbox{ iff } [x] \in E_{Q'}[\{ [y] \mid \mathfrak{F}^\natural, [y] \vDash_{v'} \psi \}]\\
 &\mbox{ iff } [x] \in Q'[\{ [y] \mid \mathfrak{F}^\natural, [y] \vDash_{v'} \psi \}]\\
 &\mbox{ iff } x \in Q[\{ y \mid \mathfrak{F}^\natural, [y] \vDash_{v'} \psi \}]\\
 &\mbox{ iff } x \in Q[\{ y \mid \mathfrak{F}, y \vDash_v \psi^\natural \}] \\ 
 &\mbox{ iff } (\exists y \in X)(yQx \; \& \; \mathfrak{F}, y \vDash_v \psi^\natural) \\
 &\mbox{ iff } \mathfrak{F}, x \vDash_v \diap \psi^\natural \\
 &\mbox{ iff } \mathfrak{F}, x \vDash_v (\exists \psi)^\natural.
\end{align*}

(3). The proof is analogous to that of Proposition~\ref{prop:skeleton of ms4-frames}(3).

(4). Let $\mathfrak{G}=(X,R,Q)$ be an $\mipc$-frame. It is clear from the definition of $\ts$-frames that $\mathfrak{G}$ is also a $\ts$-frame. Since $R$ is a partial order, $\sim$ is the identity relation. Therefore, $\mathfrak{G}$ is isomorphic to $\mathfrak{G}^\natural$.
\end{proof}

\begin{theorem}\label{thm:()^natural full and faithful}
The translation $(-)^\natural$ of $\mipc$ into $\ts$ is full and faithful; that is,
\[
\mipc \vdash\varphi \mbox{ iff } \ts \vdash\varphi^\natural.
\]
\end{theorem}

\begin{proof}
To prove faithfulness, suppose that $\ts \nvdash\varphi^\natural$. By Theorem~\ref{thm:relational semantics for ts4}, there is a $\ts$-frame $\mathfrak{F}$ such that $\mathfrak{F} \nvDash \varphi^\natural$. By Proposition~\ref{prop:skeleton of ts4-frames}, $\mathfrak{F}^\natural$ is an $\mipc$-frame and $\mathfrak{F}^\natural \nvDash \varphi$. Thus, by Theorem~\ref{thm:relational semantics for mipc}, $\mipc \nvdash \varphi$. For fullness, if $\mipc \nvdash \varphi$, then there is an $\mipc$-frame $\mathfrak{G}$ such that $\mathfrak{G} \nvDash \varphi$. By Proposition~ \ref{prop:skeleton of ts4-frames}(4), $\mathfrak{G}$ is also a $\ts$-frame and it is isomorphic to $\mathfrak{G}^\natural$. Therefore, $\mathfrak{G}^\natural \nvDash \varphi$. Proposition~\ref{prop:skeleton of ts4-frames}(3) then yields that $\mathfrak{G} \nvDash \varphi^\natural$. Thus, $\ts \nvdash \varphi^\natural$.
\end{proof}

\section{Translations into $\mst$} \label{sec:translations into mst}

In Sections~\ref{sec:godel translation} and~\ref{sec:TS4} we described full and faithful translations of $\mipc$ into $\ms$ and $\ts$, respectively. This yields the following diagram.

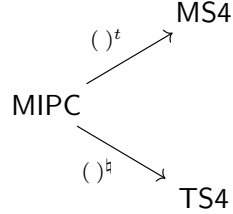
\[
\begin{tikzcd}
& \ms  & \\
\mipc \arrow[ru, " ( \; )^t"] \arrow[rd, "(\; )^\natural"'] & & \\
& \ts  &
\end{tikzcd}
\]

There does not appear to be a natural way to translate $\ms$ into $\ts$ or vice versa (see \cite{CarThesis} for details). The aim of this section is to define a new tense system and show that both $\ms$ and $\ts$ embed fully and faithfully into it, thus completing the above diagram.

\subsection{$\mst$}

Let $\mathcal L_{T\forall}$ be the propositional language with the tense modalities $\boxfr$ and $\boxpr$, and the monadic modality $\forall$. In order to stress that the language $\mathcal L_{T\forall}$ is different from $\mathcal{ML}$ and $\ts$, we use different symbols for the tense modalities.

\begin{definition}
The \emph{tense $\ms$}, denoted $\mst$, is the least classical multimodal logic containing the $\st$-axioms for $\boxfr$ and $\boxpr$, the
$\sfive$-axioms for $\forall$, the left commutativity axiom
\[
\boxfr \forall p \to \forall \boxfr p,
\]
and closed under modus ponens, substitution, and the necessitation rules (for $\boxfr$, $\boxpr$, and $\forall$).
\end{definition}

\begin{remark}\label{rem:monadic}
We can think of $\mst$ as the tense extension of $\ms$. It is worth stressing that $\mst$ is not the monadic fragment of the standard predicate extension $\qsfourt$ of $\st$. To see this, it is well known that the Barcan formula $\forall x \boxfr \varphi \to \boxfr \forall x \varphi$ and the converse Barcan formula $\boxfr \forall x \varphi \to \forall x \boxfr \varphi$ are both theorems of any tense predicate logic, hence of $\qsfourt$ as well.
Thus, the monadic fragment of $\qsfourt$ contains both the left commutativity axiom $\boxfr \forall p \to \forall \boxfr p$ and the right commutativity axiom $\forall \boxfr p \to \boxfr \forall p$. On the other hand, it is easy to see (e.g., using the Kripke semantics for $\mst$ which we will define shortly) that, while $\mst$ contains the left commutativity axiom, the right commutativity axiom is not provable in $\mst$.
\end{remark}

Algebraic semantics for $\mst$ is given by $\mst$-algebras.

\begin{definition}
An \emph{$\mst$-algebra} is a tuple $\mathfrak B=(B, \boxfr, \boxpr, \forall)$ where $(B, \boxfr, \boxpr)$ is an $\st$-algebra and $(B, \boxfr, \forall)$ is an $\ms$-algebra.
\end{definition}

As usual, the Lindenbaum-Tarski construction yields that $\mst$ is sound and complete with respect to $\mst$-algebras.

As with $\sf S4$ and $\sf S4.t$, we have that $\mst$-frames are simply $\ms$-frames, the difference is in interpreting tense modalities.
Thus, the following lemma is straightforward.

\begin{lemma} \label{prop: X^+ for ms4.t}
If $\mathfrak F=(X,R,E)$ is an $\mst$-frame, then $\mathfrak F^+:=(\wp(X), \Box_R, \boxrinv, \forall_E)$ is an $\mst$-algebra.
\end{lemma}

We next prove that each $\mst$-algebra is represented as a subalgebra of $\mathfrak F^+$ for some $\mst$-frame $\mathfrak F$. For an $\mst$-algebra $(B, \boxfr, \boxpr, \forall)$ let $H_F$, $H_P$, and $B_0$ be the $\boxfr$-fixpoints, $\boxpr$-fixpoints, and $\forall$-fixpoints, respectively. Clearly $H_F$ and $H_P$ are Heyting algebras and $B_0$ is a boolean subalgebra of $B$.

\begin{definition} \label{def:canonical frame mst}
Let $\mathfrak B=(B, \boxfr, \boxpr, \forall)$ be an $\mst$-algebra. The \emph{canonical frame} of $\mathfrak B$ is the frame
$\mathfrak B_+=(X_{\mathfrak B},R_{\mathfrak B},E_{\mathfrak B})$ where
$X_{\mathfrak B}$ is the set of ultrafilters of $B$, $x R_{\mathfrak B} y$ iff
$x\cap H_F \subseteq y$ iff $y\cap H_P \subseteq x$, and $x E_{\mathfrak B} y$ iff $x\cap B_0 = y\cap B_0$.
\end{definition}

Since $\mst$-frames are $\ms$-frames, the next lemma is obvious.

\begin{lemma}\label{lem:canonical frame for mst}
If $\mathfrak B$ is an $\mst$-algebra, then $\mathfrak B_+$ is an $\mst$-frame.
\end{lemma}

Thus, since $\beta:B \to \wp(X_{\mathfrak B})$ is an embedding of $\st$-algebras and $\ms$-algebras, we obtain the following representation theorem for $\mst$-algebras.

\begin{proposition} \label{prop: representation for mst}
Each $\mst$-algebra $\mathfrak{B}$ is isomorphic to a subalgebra of $(\mathfrak B_+)^+$.
\end{proposition}

\begin{remark}\label{rem:facts about dualities ms4.t}
To recover the image of the embedding of $\mathfrak{B}$ into $(\mathfrak B_+)^+$ we need to endow $\mathfrak B_+$ with a Stone topology. This leads to the notion of perfect $\mst$-frames and a duality between the categories of $\mst$-algebras and perfect $\mst$-frames (see \cite{CarThesis} for details).
When $\mathfrak{B}$ is finite, its embedding into $(\mathfrak B_+)^+$ is an isomorphism, and hence the categories of finite $\mst$-algebras and finite $\mst$-frames are dually equivalent.
\end{remark}

By Lemmas~\ref{prop: X^+ for ms4.t} and~\ref{lem:canonical frame for mst}, if $\mathfrak{B}$ is an $\mst$-algebra, then so is $(\mathfrak B_+)^+$. As an immediate consequence, we obtain:

\begin{corollary}
$\mst$ is canonical.
\end{corollary}

A \textit{valuation} on an $\mst$-frame $\mathfrak F=(X,R,E)$ is a map $v$ associating to each propositional letter of $\mathcal L_{T\forall}$ a subset of $\mathfrak{F}$. The boolean connectives are interpreted as usual, and
\begin{equation*}
\begin{array}{l c l}
\mathfrak{F}, x \vDash_v \boxfr \varphi & \text{ iff } & (\forall y \in X) (x R y \Rightarrow y \vDash_v \varphi ), \\
\mathfrak{F}, x \vDash_v \boxpr \varphi & \text{ iff } & (\forall y \in X) (y R x \Rightarrow y \vDash_v \varphi ), \\
\mathfrak{F}, x \vDash_v \forall \varphi & \text{ iff } & (\forall y \in X) (x E y \Rightarrow y \vDash_v \varphi).
\end{array}
\end{equation*}

\begin{theorem} \label{thm:relational semantics for mst}
$\mst \vdash \varphi$ iff $\mathfrak F \vDash \varphi$ for every $\mst$-frame $\mathfrak F$.
\end{theorem}

\begin{proof}
Soundness is a consequence of the soundness of the relational semantics for $\ms$ and $\st$. Completeness follows from the algebraic completeness and the representation of $\mst$-algebras (see Proposition~\ref{prop: representation for mst}).
\end{proof}

In Section~\ref{sec:FMP} we will prove that $\mst$ has the fmp.

\subsection{Translations of $\ts$ and $\ms$ into $\mst$}

We next define two full and faithful translations $(-)^\#: \ms \to \mst$ and $(-)^\dagger: \ts \to \mst$. 
The translation of $\ms$ into $\mst$ will reflect that $\mst$ is the tense extension of $\ms$.

\begin{definition}
We define the translation $(-)^\# : \ms \to \mst$ by replacing in each formula $\varphi$ of $\mathcal{L}_{\Box \forall}$ every occurrence of $\Box$ with $\boxfr$.
\end{definition}

\begin{theorem}\label{thm:from ms4 to ms4.t full and faithful}
The translation $(-)^\#$ of $\ms$ into $\mst$ is full and faithful; that is,
\begin{equation*}
\begin{array}{l c l}
\ms \vdash \varphi & \mbox{ iff } & \mst \vdash \varphi^\#.
\end{array}
\end{equation*}
\end{theorem}

\begin{proof}
By definition, $\mst$-frames are $\ms$-frames and valuations on $\ms$-frames and $\mst$-frames coincide. The boolean connectives and monadic modality $\forall$ are interpreted the same way in $\ms$-frames and $\mst$-frames. Also, the interpretation of $\Box$ in $\ms$-frames coincides with the interpretation of $\boxfr$ in $\mst$-frames. This implies that for each frame $\mathfrak{F}=(X,R,E)$, valuation $v$, and $x \in X$, we have
$\mathfrak{F}, x \vDash \varphi$ iff $\mathfrak{F}, x \vDash \varphi^\#$
for every $\mathcal{L}_{\Box \forall}$-formula $\varphi$.
The result then follows from the soundness and completeness of $\ms$ and $\mst$ with respect to their relational semantics (see Theorems~\ref{thm:relational semantics for ms4} and~\ref{thm:relational semantics for mst}).
\end{proof}

\begin{definition}\label{def:()^dagger}
Define the translation $(-)^\dagger: \ts \to \mst$ by
\begin{equation*}
\begin{split}
p^\dagger &=p \quad \mbox{ for each propositional letter $p$}\\
(-)^\dagger & \mbox{ commutes with the boolean connectives}\\
(\Box \varphi)^\dagger &=\boxfr \varphi^\dagger\\
(\boxf \varphi)^\dagger &=\boxfr \forall \varphi^\dagger\\
(\boxp \varphi)^\dagger &=\forall \boxpr \varphi^\dagger.
\end{split}
\end{equation*}
\end{definition}

\begin{definition}
For an $\mst$-frame $\mathfrak F=(X,R,E)$ we define $\mathfrak{F}^\dagger = (X,R,Q_E)$.
\end{definition}

\begin{proposition}\label{prop:F^dagger}
\begin{enumerate}
\item[]
\item If $\mathfrak F$ is an $\mst$-frame, then $\mathfrak{F}^\dagger$ is a $\ts$-frame.
\item Each valuation $v$ on $\mathfrak{F}$ is also a valuation on $\mathfrak{F}^\dagger$ such that for each $x \in \mathfrak{F}$ and $\mathcal{ML}$-formula $\varphi$, we have
\[
\mathfrak{F^\dagger}, x \vDash_v \varphi \mbox{ iff } \mathfrak{F},x \vDash_v \varphi^\dagger.
\]
\item For each $\mathcal{ML}$-formula $\varphi$, we have
\[
\mathfrak{F}^\dagger \vDash \varphi \mbox{ iff } \mathfrak{F} \vDash \varphi^\dagger.
\]
\item For any $\ts$-frame $\mathfrak{G}$ there is an $\mst$-frame $\mathfrak{F}$ such that $\mathfrak{G}=\mathfrak{F}^\dagger$.
\end{enumerate}
\end{proposition}

\begin{proof}
(1). Since $\mst$-frames coincide with $\ms$-frames, we already observed in Remark~\ref{rem:correspondence mipc, ms and ts frames}(2) that $\mathfrak{F}^\dagger$ is a $\ts$-frame.

(2).
It is clear that if $v$ is a valuation on $\mathfrak{F}$, then $v$ is also a valuation on $\mathfrak{F}^\dagger$. We show that $\mathfrak{F}^\dagger, x \vDash_v \varphi$ iff $\mathfrak{F},x \vDash_v \varphi^\dagger$ by induction on the complexity of $\varphi$. The only nontrivial cases are when $\varphi$ is of the form $\Box \psi$, $\boxf \psi$ and $\boxp \psi$. Suppose $\varphi=\Box \psi$. Then, by the induction hypothesis,
\begin{align*}
\mathfrak{F}^\dagger, x \vDash_v \Box \psi &\mbox{ iff } (\forall y \in X)(xRy \, \Rightarrow \, \mathfrak{F}^\dagger, y \vDash_v \psi)\\
 &\mbox{ iff } (\forall y \in X)(xRy \, \Rightarrow \, \mathfrak{F}, y \vDash_v \psi^\dagger) \\
 &\mbox{ iff } \mathfrak{F}, x \vDash_v \boxfr \psi^\dagger\\
 &\mbox{ iff } \mathfrak{F}, x \vDash_v (\Box \psi)^\dagger.
\end{align*}
Suppose $\varphi=\boxf \psi$. Then, by the induction hypothesis,
\begin{align*}
\mathfrak{F}^\dagger, x \vDash_v \boxf \psi &\mbox{ iff } (\forall y \in X)(x Q_E y \, \Rightarrow \, \mathfrak{F}^\dagger, y \vDash_v \psi)\\
 &\mbox{ iff } (\forall z \in X)(xRz \, \Rightarrow \, (\forall y \in X)(zEy \, \Rightarrow \, \mathfrak{F}^\dagger, y \vDash_v \psi))\\
  &\mbox{ iff } (\forall z \in X)(xRz \, \Rightarrow \, (\forall y \in X)(zEy \, \Rightarrow \, \mathfrak{F}, y \vDash_v \psi^\dagger)) \\
  &\mbox{ iff } (\forall z \in X)(xRz \, \Rightarrow \, \mathfrak{F}, z \vDash \forall \psi^\dagger)\\
 &\mbox{ iff } \mathfrak{F}, x \vDash_v \boxfr \forall \psi^\dagger\\
 &\mbox{ iff } \mathfrak{F}, x \vDash_v (\boxf \psi)^\dagger.
\end{align*}
Suppose $\varphi=\boxp \psi$. Then, by the induction hypothesis,
\begin{align*}
\mathfrak{F}^\dagger, x \vDash_v \boxp \psi &\mbox{ iff } (\forall y \in X)(y Q_E x \, \Rightarrow \, \mathfrak{F}^\dagger, y \vDash_v \psi)\\
 &\mbox{ iff } (\forall y,z \in X)(yRz \ \& \ zEx \, \Rightarrow \, \mathfrak{F}^\dagger, y \vDash_v \psi)\\
 &\mbox{ iff } (\forall z \in X)(zEx \, \Rightarrow \, (\forall y \in X)(yRz \, \Rightarrow \, \mathfrak{F}^\dagger, y \vDash_v \psi))\\
  &\mbox{ iff } (\forall z \in X)(zEx \, \Rightarrow \, (\forall y \in X)(yRz \, \Rightarrow \, \mathfrak{F}, y \vDash_v \psi^\dagger)) \\
  &\mbox{ iff } (\forall z \in X)(zEx \, \Rightarrow \, \mathfrak{F}, z \vDash \boxpr \psi^\dagger)\\
  &\mbox{ iff } (\forall z \in X)(xEz \, \Rightarrow \, \mathfrak{F}, z \vDash \boxpr \psi^\dagger)\\
 &\mbox{ iff } \mathfrak{F}, x \vDash_v \forall \boxpr \psi^\dagger\\
 &\mbox{ iff } \mathfrak{F}, x \vDash_v (\boxp \psi)^\dagger.
\end{align*}

(3). The proof that $\mathfrak{F}^\dagger \vDash \varphi$ iff $\mathfrak{F} \vDash \varphi^\dagger$ is analogous to that of Proposition~\ref{prop:skeleton of ms4-frames}(3).

(4). Let $\mathfrak{G}=(X,R,Q)$ be a $\ts$-frame. As we observed in Remark~\ref{rem:correspondence mipc, ms and ts frames}, $\mathfrak{F}=(X,R,E_Q)$ is an $\ms$-frame, and so an $\mst$-frame.
By definition of $\ts$-frames we have that $Q=Q_{E_Q}$, and hence $\mathfrak{G}=\mathfrak{F}^\dagger$.
\end{proof}

\begin{theorem} \label{thm:()^dagger full and faithful}
The translation $(-)^\dagger$ of $\ts$ into $\mst$ is full and faithful; that is,
\begin{equation*}
\begin{array}{l c l}
\ts \vdash \varphi & \mbox{ iff } & \mst \vdash \varphi^\dagger.
\end{array}
\end{equation*}
\end{theorem}

\begin{proof}
To prove faithfulness, suppose that $\mst \nvdash\varphi^\dagger$. By Theorem~\ref{thm:relational semantics for mst}, there is an $\mst$-frame $\mathfrak{F}$ such that $\mathfrak{F} \nvDash \varphi^\dagger$. By Proposition~\ref{prop:F^dagger}, $\mathfrak{F}^\dagger$ is a $\ts$-frame and $\mathfrak{F}^\dagger \nvDash \varphi$. Thus, $\ts \nvdash \varphi$ by Theorem~\ref{thm:relational semantics for ts4}. For fullness, if $\ts \nvdash \varphi$, then there is a $\ts$-frame $\mathfrak{G}$ such that $\mathfrak{G} \nvDash \varphi$. By Proposition~\ref{prop:F^dagger}(4), there is an $\mst$-frame $\mathfrak{F}$ such that $\mathfrak{G}$ is isomorphic to $\mathfrak{F}^\dagger$. Therefore, $\mathfrak{F}^\dagger \nvDash \varphi$. Proposition~\ref{prop:F^dagger}(3) then implies that $\mathfrak{F} \nvDash \varphi^\dagger$. Thus, $\mst \nvdash \varphi^\dagger$.
\end{proof}

\begin{remark}\label{rem:justification dagger and flat}
\begin{enumerate}
\item[]
\item The definition of the translation $(-)^\dagger: \ts \to \mst$ is suggested by the correspondence between $\ts$-frames and $\mst$-frames. 
Indeed, given an $\mst$-frame $\mathfrak{F}$, the relation $Q_E$ in $\mathfrak{F}^\dagger$ is the composition of $R$ and $E$, and the inverse relation $\qeinv$ is the composition of $E$ and $\rinv$. Therefore, the modalities $\boxf$ and $\boxp$ are translated as 
$\boxfr \forall$ and $\forall \boxpr$, respectively.
\item 
It is natural to consider a modification of $(-)^\dagger$ where $\boxp$ is translated as $\boxpr \forall$. However, such a modification
is neither full nor faithful. Nevertheless, its 
composition with $(-)^\natural:\mipc \to \ts$ is full and faithful, as we will 
see at the end of Section~4.3.
\end{enumerate}
\end{remark}

\subsection{Translations of $\mipc$ into $\mst$}

We denote the composition of $(-)^\#$ and $(-)^t$ by $(-)^{t \#}$, and the composition of $(-)^\dagger$ and $(-)^\natural$ by $(-)^{\natural \dagger}$. Since we proved that all these four translations are full and faithful, we also have that $(-)^{t \#}$ and $(-)^{\natural \dagger}$ are full and faithful translations of $\mipc$ into $\mst$. We have thus obtained the following diagram of full and faithful translations. We next show that this diagram is commutative up to logical equivalence in $\mst$.
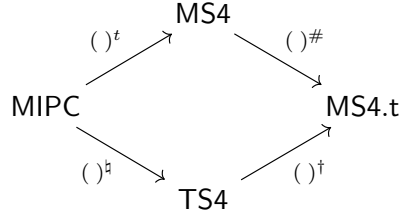
\[
\begin{tikzcd}
& \ms \arrow[rd, "( \; )^\#"] & \\
\mipc \arrow[ru, " ( \; )^t"] \arrow[rd, "(\; )^\natural"'] & & \mst \\
& \ts \arrow[ru, "(\; )^\dagger"'] &
\end{tikzcd}
\]

\begin{lemma}\label{lem:tsharp in mst is an upset}
For any formula $\varphi$ of $\mathcal L_{\forall\exists}$, we have
\[
\mst \vdash \varphi^{t \#} \leftrightarrow \diapr \varphi^{t \#}.
\]
\end{lemma}

\begin{proof}
By Lemma~\ref{lem:translation in ms4 is upset} and Theorem~\ref{thm:from ms4 to ms4.t full and faithful}, $\mst \vdash \varphi^{t \#} \to \boxfr \varphi^{t \#}$. Therefore, $\mst \vdash \diapr \varphi^{t \#} \to \diapr \boxfr \varphi^{t \#}$. The tense axiom then gives $\mst \vdash \diapr \varphi^{t \#} \to \varphi^{t \#}$.
Thus, $\mst \vdash \varphi^{t \#} \leftrightarrow \diapr \varphi^{t \#}$.
\end{proof}

\begin{theorem}\label{thm:commutativity of diagram}
For any $\mathcal L_{\forall\exists}$-formula $\chi$ we have
\[
\mst \vdash \chi^{t \#} \leftrightarrow \chi^{\natural \dagger}.
\]
\end{theorem}

\begin{proof}
The two compositions compare as follows:
\begin{align*}
&\bot^{t \#}=\bot  && \bot^{\natural \dagger}=\bot \\
&p^{t \#}=\boxfr p  && p^{\natural \dagger}=\boxfr p \\
&(\varphi \wedge \psi)^{t \#}= \varphi^{t \#} \wedge \psi^{t \#} \qquad && (\varphi \wedge \psi)^{\natural \dagger}=
\varphi^{\natural \dagger} \wedge \psi^{\natural \dagger} \\
&(\varphi \vee \psi)^{t \#}= \varphi^{t \#} \vee \psi^{t \#} \qquad && (\varphi \vee \psi)^{\natural \dagger}=
\varphi^{\natural \dagger} \vee \psi^{\natural \dagger} \\
&(\varphi \to \psi)^{t \#}= \boxfr(\neg \varphi^{t \#} \vee \psi^{t \#}) \qquad && (\varphi \to \psi)^{\natural \dagger}=
\boxfr(\neg \varphi^{\natural \dagger} \vee \psi^{\natural \dagger}) \\
&(\forall \varphi)^{t \#}= \boxfr \forall \varphi^{t \#} && (\forall \varphi)^{\natural \dagger}= \boxfr \forall \varphi^{\natural \dagger} \\
&(\exists \varphi)^{t \#}= \exists \varphi^{t \#} && (\exists \varphi)^{\natural \dagger}= (\diap \varphi^\natural)^\dagger=
(\neg \boxp \neg \varphi^\natural)^\dagger \\
& && \hphantom{(\exists \varphi)^{\natural \dagger}} = \neg \forall \boxpr \neg \varphi^{\natural \dagger}
\end{align*}
Thus, they are identical except the $\exists$-clause. Therefore, to prove that $\mst \vdash \chi^{t \#} \leftrightarrow \chi^{\natural \dagger}$ it is sufficient to prove that $\mst \vdash \varphi^{t \#} \leftrightarrow \varphi^{\natural \dagger}$ implies $\mst \vdash \exists \varphi^{t \#} \leftrightarrow \neg \forall \boxpr \neg \varphi^{\natural \dagger}$.
Since $\mst \vdash \neg \forall \boxpr \neg \varphi^{\natural \dagger} \leftrightarrow \exists \diapr \varphi^{\natural \dagger}$, it is enough to prove that $\mst \vdash \exists \varphi^{t \#} \leftrightarrow \exists \diapr \varphi^{\natural \dagger}$.
From the assumption $\mst \vdash \varphi^{t \#} \leftrightarrow \varphi^{\natural \dagger}$ it follows that $\mst \vdash \exists \diapr \varphi^{t \#} \leftrightarrow \exists \diapr \varphi^{\natural \dagger}$. By Lemma~\ref{lem:tsharp in mst is an upset}, $\mst \vdash \varphi^{t \#} \leftrightarrow \diapr \varphi^{t \#}$ and hence $\mst \vdash \exists \varphi^{t \#} \leftrightarrow \exists \diapr \varphi^{t \#}$. Thus, $\mst \vdash \exists \varphi^{t \#} \leftrightarrow \exists \diapr \varphi^{\natural \dagger}$.
\end{proof}

As we pointed out in Remark~\ref{rem:justification dagger and flat}(2), there is another natural translation of $\mipc$ into $\mst$.

\begin{definition}\label{def:()^flat}
Let $(-)^\flat:\mipc \to \mst$ be the translation that
differs from $(-)^{t \#}$ and $(-)^{\natural \dagger}$ only in the $\exists$-clause:
\begin{align*}
&(\exists \varphi)^\flat= \diapr \exists \varphi^\flat.
\end{align*}
\end{definition}

The translation $(-)^\flat$ provides a temporal interpretation of intuitionistic monadic quantifiers that is similar to the translation $(-)^\natural$ (see also Section~6).

\begin{theorem}\label{thm:flat full and faith and equiv}
For any $\mathcal L_{\forall\exists}$-formula $\chi$ we have
\[
\mst \vdash \chi^\flat \leftrightarrow \chi^{t \#}.
\]
Consequently, the translation $(-)^\flat$ of $\mipc$ into $\mst$ is full and faithful.
\end{theorem}

\begin{proof}
The translations $()^\flat$ and $(-)^{t \#}$ are identical except the $\exists$-clause. Therefore, to prove that $\mst \vdash \chi^\flat \leftrightarrow \chi^{t \#}$ it is sufficient to prove that $\mst \vdash \varphi^\flat \leftrightarrow \varphi^{t \#}$ implies $\mst \vdash \diapr \exists \varphi^\flat \leftrightarrow \exists \varphi^{t \#}$.
By Lemma~\ref{lem:tsharp in mst is an upset}, $\mst \vdash (\exists \varphi)^{t \#} \leftrightarrow \diapr (\exists \varphi)^{t \#}$ which means $\mst \vdash \exists \varphi^{t \#} \leftrightarrow  \diapr \exists \varphi^{t \#}$. From the assumption $\mst \vdash \varphi^\flat \leftrightarrow \varphi^{t \#}$ it follows that $\mst \vdash \diapr \exists \varphi^\flat \leftrightarrow \diapr \exists \varphi^{t \#}$. Thus, $\mst \vdash  \diapr \exists \varphi^\flat \leftrightarrow \exists \varphi^{t \#}$. Since $(-)^{t \#}$ is full and faithful, it follows that $(-)^\flat$ is full and faithful as well.
\end{proof}

As a result, we obtain the following diagram of full and faithful translations that is commutative up to logical equivalence in $\mst$.
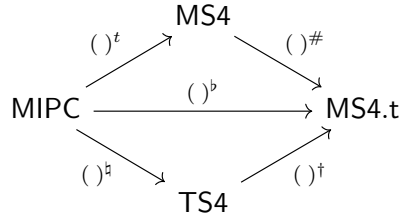
\[
\begin{tikzcd}
& \ms \arrow[rd, "( \; )^\#"] & \\
\mipc \arrow[ru, " ( \; )^t"] \arrow[rd, "(\; )^\natural"']  \arrow[rr,  " ( \; )^\flat"]& & \mst \\
& \ts \arrow[ru, "(\; )^\dagger"'] &
\end{tikzcd}
\]

\section{Finite Model Property}\label{sec:FMP}

In this
section we prove that the logics studied in this paper all have the fmp. Our strategy is to first establish the fmp for $\mst$, and then use the full and faithful translations to conclude that all the logics we have considered have the fmp.

Let $\mathfrak{B}=(B, \boxfr, \boxpr, \forall)$ be an $\mst$-algebra and $S \subseteq B$ a finite subset. Then $(B, \forall)$ is an $\sfive$-algebra.
Let $(B', \forall')$ be the $\sfive$-subalgebra of $(B, \forall)$ generated by $S$. It is well known (see \cite{Bas58}) that $(B', \forall')$ is finite. Define $\boxfrdot$ and $\boxprdot$ on $B'$ by
\begin{align*}
\boxfrdot a &= \bigvee \{ b \in B'\cap H_F \mid b \leq a \} \\
\boxprdot a &= \bigvee \{ b \in B'\cap H_P \mid b \leq a \}.
\end{align*}
We denote $(B', \boxfrdot, \boxprdot, \forall')$ by $\mathfrak{B}_S$.

\begin{lemma} \label{lem:boxfdot and boxpdot}
$\mathfrak{B}_S$ is an $\mst$-algebra.
\end{lemma}

\begin{proof}
By definition, $(B',\forall')$ is an $\sfive$-algebra. Since $(B, \boxfr)$ and $(B, \boxpr)$ are both $\sfour$-algebras, a standard argument (see \cite[Lem.~4.14]{MT44}) shows that $(B', \boxfrdot)$ and $(B', \boxprdot)$ are also $\sfour$-algebras. We show that $(B', \boxfrdot, \boxprdot)$ is an $\st$-algebra.
Let $H_F$ be the algebra of $\boxfr$-fixpoints and $H_P$ the algebra of $\boxpr$-fixpoints of $\mathfrak{B}$. As noted in Remark~\ref{rem:dual iso H_F and H_P}, $\neg$ is a dual isomorphism between $H_F$ and $H_P$. Therefore,
\begin{align*}
\diafrdot  a := \neg \boxfrdot \neg a & = \neg \bigvee \{ b \in B'\cap H_F \mid b \leq \neg a \} \\
& = \neg \bigvee \{  b \in B'\cap H_F \mid a \leq \neg b  \} \\
& = \bigwedge \{ \neg b \mid b \in B'\cap H_F, \; a \leq \neg b  \} \\
& = \bigwedge \{ c \in B' \cap H_P \mid  a \leq c  \}.
\end{align*}
Since this meet is finite and $\boxpr$ commutes with finite meets, we obtain
\begin{align*}
\boxpr \diafrdot a & = \boxpr \left( \bigwedge \{ c \in B' \cap H_P \mid  a \leq c  \} \right) \\
& = \bigwedge \{ \boxpr c \mid c \in B' \cap H_P , \; a \leq c \}  \\
& = \bigwedge \{ c \in B' \cap H_P \mid  a \leq c  \} \\
& = \diafrdot  a.
\end{align*}
Thus, $\diafrdot a \in B' \cap H_P $ which yields
\[
\boxprdot \diafrdot  a = \bigvee \{ b \in B' \cap H_P \mid b \leq \diafrdot a \} = \diafrdot  a.
\]
Similarly, we have that $\diaprdot  a = \bigwedge \{ c \in B' \cap H_F \mid  a \leq c \}$ from which we deduce that $\boxfrdot \diaprdot a = \diaprdot a$. This implies that $a \le \boxprdot \diafrdot  a$ and $a \le \boxfrdot \diaprdot  a$. Consequently, $(B, \boxfrdot, \boxprdot)$ is an $\st$-algebra.

It remains to show that $\boxfrdot \forall' a \le \forall' \boxfrdot a$ holds in $\mathfrak{B}_S$. For this it is sufficient to show that the set $B_0':=B' \cap B_0$ of the $\forall'$-fixpoints of $B'$ is an $\sfour$-subalgebra of $(B', \boxfrdot)$ because then $\boxfrdot \forall' a=\forall' \boxfrdot \forall' a \le \forall' \boxfrdot a$. Suppose that $d \in B_0'$. Then $\boxfrdot d = \bigvee \{ b \in B'\cap H_F \mid b \leq d \}$. Let $b \in B'\cap H_F$. By Lemma~\ref{lem:equivalent axioms ms-alg}, $\exists b = \exists \boxfr b =\boxfr \exists \boxfr b = \boxfr \exists b$. Therefore, $\exists b \in B' \cap H_F$. Moreover, $b \le \exists b$ and $b \le d$ implies $\exists b \le \exists d = d$. Thus, $\boxfrdot d = \bigvee \{ \exists b \mid b \in B'\cap H_F, \; b \leq d \}$.
Since $(B', \forall')$ is an $\sfive$-algebra, $B_0'$ is the set of $\exists'$-fixpoints of $B'$ and is closed under finite joins. Consequently, $\boxfrdot d \in B_0'$.
\end{proof}

\begin{theorem}\label{thm:fmp of mst}
$\mst$ has the fmp.
\end{theorem}

\begin{proof}
It is sufficient to prove that each $\mathcal L_{T\forall}$-formula $\varphi$ refuted on some $\mst$-algebra is also refuted on a finite $\mst$-algebra.
Let $t(x_1, \ldots, x_n)$ be the term in the language of $\mst$-algebras that corresponds to $\varphi$, and suppose there is an $\mst$-algebra $\mathfrak{B}=(B, \boxfr, \boxpr, \forall)$ and $a_1, \ldots, a_n \in B$ such that $t(a_1, \ldots, a_n) \neq 1$ in $\mathfrak{B}$. Let
\[
S=\{ t'(a_1, \ldots, a_n) \mid t' \mbox{ is a subterm of } t \}.
\]
Then $S$ is a finite subset of $B$. Therefore, by Lemma~\ref{lem:boxfdot and boxpdot}, $\mathfrak{B}_S=(B', \boxfrdot, \boxprdot, \forall)$ is a finite $\mst$-algebra. It follows from the definition of $\boxfrdot$ that, for each $b \in B'$, if $\boxfr b \in B'$, then $\boxfrdot b= \boxfr b$. Similarly, if
$\boxpr b \in B$, then $\boxprdot b= \boxpr b$. Thus, for each subterm $t'$ of $t$, the computation of $t'$ in $\mathfrak{B}_S$ is the same as that in $\mathfrak{B}$. Consequently, $t(a_1, \ldots, a_n) \neq 1$ in $\mathfrak{B}_S$, and we have found a finite $\mst$-algebra refuting $\varphi$.
\end{proof}

\begin{remark}
Lemma~\ref{lem:boxfdot and boxpdot} in particular proves that $\mathfrak{B}_S$ is an $\st$-algebra. Thus, the proof of the fmp for $\mst$ contains the proof of the fmp for $\st$.
In fact, $\mst$ is a conservative extension of $\st$.
\end{remark}

We conclude this section by showing that the fmp for $\ts$,
$\ms$, and $\mipc$ is a consequence of Theorem~\ref{thm:fmp of mst}.

\begin{theorem}
\begin{enumerate}
\item[]
\item $\ts$ has the fmp.
\item $\ms$ has the fmp.
\item $\mipc$ has the fmp.
\end{enumerate}
\end{theorem}

\begin{proof}
(1). Suppose that $\ts \nvdash \varphi$. By Theorem~\ref{thm:()^dagger full and faithful}, $\mst \nvdash \varphi^\dagger$. Since $\mst$ has the fmp, there is a finite $\mst$-algebra $\mathfrak{B}$ such that $\mathfrak{B} \nvDash \varphi^\dagger$. As noted in Remark~\ref{rem:facts about dualities ms4.t}, $\mathfrak{B}$ is isomorphic to $(\mathfrak B_+)^+$. This yields that $\mathfrak{B}_+ \nvDash \varphi^\dagger$. By Proposition~\ref{prop:F^dagger}(2), $(\mathfrak{B}_+)^\dagger \nvDash \varphi$. We have thus obtained a finite $\ts$-frame $(\mathfrak{B}_+)^\dagger$ refuting $\varphi$. So $((\mathfrak{B}_+)^\dagger)^+$ is a finite $\ts$-algebra such that $((\mathfrak{B}_+)^\dagger)^+ \nvDash \varphi$.

(2). Similar to the proof of (1) but uses the translation $(-)^\# : \ms \to \mst$ instead of $(-)^\dagger$.

(3). Similar to the proof of (1) but uses the composition $(-)^{t \#}: \mipc \to \mst$ instead of $(-)^\dagger$. Alternatively, we can use the other translations $(-)^{\natural \dagger}$ and $(-)^\flat$ of $\mipc$ into $\mst$.
\end{proof}

\section{Connection with the full predicate case}

In \cite{BC20b} we studied a temporal translation of the predicate intuitionistic logic $\iqc$ that is the predicate analogue of the translation $(-)^\flat$ of Definition~\ref{def:()^flat}.
We proved that this translation embeds
$\iqc$ fully and faithfully into a weakening of the tense predicate logic $\qst$. This weakening is necessary since $\qst$ proves the Barcan formula for both $\Box_F$ and $\Box_P$, so Kripke frames of $\qst$ have constant domains, and hence they validate the translation of the constant domain axiom $\forall x(A\vee B)\to (A\vee \forall x B)$, where $x$ is not free in $A$. Since this is not provable in $\iqc$, the translation cannot be full.
Instead we considered the tense predicate logic $\qost$ in which the universal instantiation axiom $\forall x A \to A(y/x)$ is replaced by its weakened version $\forall y(\forall x A \to A(y/x))$. The main result of \cite{BC20b} proves that $\iqc$ translates fully and faithfully into $\qost$ (provided the translation is restricted to sentences).

It is natural to investigate the relationship between $\mst$ and predicate extensions of $\st$.
As we already pointed out in Remark~\ref{rem:monadic}, $\mst$ is not the monadic fragment of $\qst$.
In addition, $\mst$ cannot be the monadic fragment of $\qost$ either since the formula $\forall x A \to A$ is not in general provable in $\qost$,
whereas $\forall \varphi \to \varphi$ is provable in $\mst$.
On the other hand, call a formula $\varphi$ (in the language of $\mst$) {\em bounded} if each occurrence of a propositional letter in $\varphi$ is under the scope of $\forall$. Bounded formulas play the same role as sentences of $\qost$ containing only one fixed variable. It is quite plausible that for a bounded formula $\varphi$ we have $\mst\vdash\varphi$ iff $\qost$ proves the translation of $\varphi$ where each occurrence of a propositional letter $p$ is replaced with the unary predicate $P(x)$ and $\forall$ is replaced with $\forall x$ (for a similar translation of $\sf MIPC$ and its extensions into $\iqc$ and its extensions, see \cite{Ono87}). If true, this would yield that the monadic sentences provable in $\qost$ are exactly the bounded formulas $\varphi$ provable in $\mst$.
It would also yield that restricting the translation $\iqc \to \qost$ of \cite{BC20b} to the monadic setting gives the translation
$(-)^\flat:\mipc \to \mst$ for bounded formulas.

It is natural to seek an axiomatization of the full monadic fragment of $\qost$. Note that in this fragment $\forall$ does not behave like an $\sfive$-modality. For example, $\forall \varphi \to \varphi$
is not in general a theorem of this fragment.

Finally, the translation $(-)^\#: \ms \to \mst$ suggests a translation of
$\qsfour$ into $\qost$ which replaces each occurrence of $\Box$ with $\boxfr$.
It is easy to see that for sentences this translation is full and faithful. 
Composing it with the standard G\"odel translation of $\iqc$ into 
$\qsfour$ yields a translation $\iqc \to \qost$ which is different from the translation of \cite{BC20b}. This translation restricts to the translation $(-)^{t \#}: \mipc \to \mst$ for bounded formulas. Thus, the upper part of the diagram of Section~4.3 extends to the predicate case. 

On the other hand, we do not see a natural way to interpret the tense modalities of $\ts$ as monadic quantifiers, and hence we cannot think of a natural predicate logic which could take the role of $\ts$ in the diagram of Section~4.3. Thus, the lower part of the diagram does not seem to have a natural extension to the predicate case.
Nevertheless, we can consider the predicate analogue of the translation $(-)^{\natural \dagger}:\mipc \to \mst$.
Arguing 
as in Theorems~\ref{thm:commutativity of diagram} and~\ref{thm:flat full and faith and equiv} yields a translation of $\iqc$ into $\qost$ that is full and faithful on sentences and coincides, up to logical equivalence in $\qost$, with the other two predicate translations described in this section. 

We thus obtain the following diagram in the predicate case which is commutative up to logical equivalence in $\qost$.
\[
\begin{tikzcd}
& \qsfour \arrow[rd, "( \; )^\#"] & \\
\iqc \arrow[ru, " ( \; )^t"] \arrow[rr, "(\; )^{\natural \dagger}"', bend right=50, start anchor={[xshift=1.5ex]}, end anchor={[xshift=-2ex]}]  \arrow[rr,  " ( \; )^\flat"]& & \qost \\
& &
\end{tikzcd}
\]

\end{document}